\documentclass[12pt]{article}
\usepackage[left=2.00cm, right=2.00cm, top=2.00cm, bottom=2.00cm]{geometry}
\usepackage[utf8]{inputenc}
\usepackage[english]{babel}
\usepackage{enumitem}
\usepackage{ifthen}
\usepackage{subcaption}
\usepackage[pagewise]{lineno} 
\usepackage[table]{xcolor}
\usepackage{todonotes}
\usepackage[colorlinks=true,citecolor=blue,linkcolor=blue, urlcolor=gray]{hyperref}
\usepackage{fontawesome5}
\usepackage{xparse}       
\usepackage{etoolbox}     
\usepackage{authblk}      
\usepackage{xspace}
\usepackage{multirow}
\usepackage{tikz}
\usetikzlibrary{calc}
\usepackage[font=small]{caption}

\usepackage{todonotes}

\usepackage{lineno}

\makeatletter
\let\@authorline\@empty    
\let\@affilblock\@empty    
\newcommand{\emailicon}[1]{\href{mailto:#1}{\textsuperscript{\faEnvelope[regular]}}}
\newcommand{\orcidicon}[1]{\href{https://orcid.org/#1}{\textsuperscript{\faOrcid}}}
\newcommand{\homepage}[1]{\href{#1}{\textsuperscript{\faHome}}}
\newcommand{\authorentry}[5]{%
  \ifx\@authorline\@empty
  \else
    \g@addto@macro\@authorline{,\ }%
  \fi
  \g@addto@macro\@authorline{%
    #1%
    \ifx&#2&\else\textsuperscript{#2}\fi%
    \ifx&#3&\else\emailicon{#3}\fi%
    \ifx&#4&\else\orcidicon{#4}\fi%
    \ifx&#5&\else\homepage{#5}\fi%
  }%
}
\newcommand{\affilentry}[2]{%
  \g@addto@macro\@affilblock{%
    \textsuperscript{#1}#2\\%
  }%
}
\renewcommand{\maketitle}{%
  \begin{center}
    {\LARGE \bfseries \@title \par}
    \vskip 1em
    {\normalsize
      \@authorline
    }
    \vskip 1em
    {\small \@affilblock}
    \vskip 1em
    {\small \@date \par}
  \end{center}
  \vskip 2em
}
\makeatother

\title{Local Distance Antimagic Labeling of Neighborhood Balanced Colored Graphs}
\authorentry{Maurice Genevieva Almeida}{1,3}{p20230078@goa.bits-pilani.ac.in}{}{}
\authorentry{Tarkeshwar Singh}{1}{tksingh@goa.bits-pilani.ac.in}{}{}
\authorentry{Andrea Semani\v{c}ov\'{a}-Fe\v{n}ov\v{c}\'ikov\'{a}}{2}{andrea.fenovcikova@tuke.sk}{}{}

\affilentry{1}{Department of Mathematics, \\
BITS Pilani, K K Birla Goa Campus, Goa, India.}
\affilentry{2}{Department of Applied Mathematics and Informatics,\\ Technical University of Ko\v{s}ice, Slovak Republic.}
\affilentry{3}{Rosary College of Commerce and Arts,\\ Navelim, Goa, India.}

\usepackage{tikz}
\usepackage{graphicx} 
\usepackage[utf8]{inputenc}
\usepackage{amsmath}
\usepackage{amsfonts}
\usepackage{url}
\usepackage{amssymb}
\usepackage{makeidx}
\usepackage{graphicx}
\usepackage{amsthm}
\usepackage{lineno}
\usepackage{xspace}
\usepackage{mathtools}
\newtheorem{theorem}{Theorem}[section]

\newtheorem{lemma}[theorem]{Lemma}

\newtheorem{proposition}[theorem]{Proposition}

\newtheorem{obs}[theorem]{Observation}
\newtheorem{cor}[theorem]{Corollary}
\newtheorem{dfn}[theorem]{Definition}

\newtheorem{problem}[theorem]{Problem}

\newcommand{\nbc}{neighborhood balanced $k$-coloring\xspace}
\newcommand{\nbcl}{neighborhood balanced $k$-colorable\xspace}

\newcommand{\nbcd}{neighborhood balanced $k$-colored\xspace}
\newcommand{\nnbc}{nearly neighborhood balanced $k$-coloring\xspace}
\newcommand{\nnbcl}{nearly neighborhood balanced $k$-colorable\xspace}

\newcommand{\nnbcd}{nearly neighborhood balanced $k$-colored graph\xspace}

\usepackage{mathtools}

\begin{document}
\allowdisplaybreaks[4]
\date{}
\maketitle
\begin{abstract}
	Let $G=(V,E)$ be a graph without isolated vertices. A bijection $f\colon V\to \{1,2,\dots,|V|\}$ is called a local distance antimagic labeling if $w(u)\not=w(v)$ for every edge $uv\in E$, where the weight $w(u)$  of a vertex $u$ is defined as the sum of labels assigned to all vertices adjacent to $u$. The local distance antimagic chromatic number of $G$, denoted by $\chi_{ld}(G)$, is defined as the minimum number of colors taken over all colorings of $G$ induced by local distance antimagic labelings. 
    In this paper, we investigate local distance antimagic labelings of the direct product and the union of certain classes of graphs, namely, neighborhood balanced 
$k$-colored graphs and nearly neighborhood balanced 
$k$-colored graphs.
\end{abstract}

	\section{Introduction}
	By a graph $G=(V, E)$, we mean a finite, simple, undirected graph having neither multiple edges nor loops. For graph theoretic terminology and notation, we refer to Chartrand and Lesniak~\cite{Chart}.

	The concept of antimagic labeling was introduced by Hartsfield and Ringel~\cite{hartsfield} in 1990. A graph $G$ is said to be {\it antimagic} if its edges  can be labeled with the integers $\{1,2,\dots,|E|\}$ in such   a way that the sums of the labels of the edges incident to each vertex (called the weight of the vertex) are all distinct. Hartsfield and Ringel conjectured that   every connected graph with at least three vertices admits an antimagic labeling. They also proposed a weaker conjecture stating that   every tree with at least three vertices admits an antimagic labeling. Both conjectures have been partially verified by several authors, but remain open in general, see Gallian~\cite{galliansurvey} for a comprehensive survey.
	
	Arumugam and Kamatchi~\cite{Kamatchi} introduced a vertex version of the antimagic labeling of a graph, defined  as follows. A bijection $f: V \to \{1,2,\dots,|V|\}$ is called a {\it distance antimagic labeling} of $G$ if all   vertices of $G$ have distinct vertex weights, where the weight of a vertex $v$ is defined by $w(v)=\sum_{x\in N(v)}f(x)$, where $N(v)$ denotes the open neighborhood of   $v$, i.e.,  the set of all vertices of   $G$  adjacent to $v$.  A graph $G$  is called a {\it distance antimagic graph} if it admits a distance antimagic labeling $f$. 
	
	
	Arumugam et al.~\cite{premalatha} introduced a local version of antimagic labeling.  A bijection $f\colon E\to \{1,2,\dots, |E|\}$ is called a {\it local antimagic labeling} of $G$ if for any two adjacent vertices $u$ and $v$, $w(u)\not=w(v)$, where $w(u)=\sum_{e\in E(u)}f(e)$, and $E(u)$ denotes the set of edges incident to $u$. Thus, any local antimagic labeling induces a proper vertex coloring of $G$, where each vertex $v$ is assigned the color $w(v)$. The {\it local antimagic chromatic number} of $G$, denoted by $\chi_{la}(G)$, is the minimum number of colors taken over all colorings of $G$ induced by local antimagic labelings. Arumugam et al.~\cite{premalatha} conjectured that every connected graph with at least three vertices admits a local antimagic labeling. This conjecture was partially verified by Bensmail et al.~\cite{bensmail}, and later completely resolved by Haslegrave~\cite{haslegrave} using probabilistic methods.
Recently, several authors have studied the local antimagic chromatic number for various classes of graphs; see, for instance, \cite{premalatha, LAM4, lauLAG1,lauLAG2,lauLAG3, ravit, raviLAG}.
	
	Motivated by the concept local antimagic labeling, Divya et al.~\cite{yamini} and Handa et al.~\cite{Handa} independently introduced the notion of local distance antimagic labeling. Let $f\colon V\to\{1,2,\dots,|V|\}$ be a bijection. For each vertex $v\in V$,  the weight of $v$ is defined by $w(v)=\sum_{x\in N(v)}f(x)$. The labeling $f$ is called a {\it local distance antimagic labeling} of $G$ if $w(u)\not=w(v)$ for every pair of adjacent vertices $u,v\in V$. A graph that admits such a labeling is called a {\it local distance antimagic graph}. Every local distance antimagic labeling induces a proper vertex coloring of the graph, where each vertex $v$ is assigned the color $w(v)$. The {\it local distance antimagic chromatic number} of $G$, denoted $\chi_{ld}(G)$, is the minimum number of colors taken over all colorings of $G$ induced by local distance antimagic labelings. Clearly, $\chi_{ld}(G) \geq \chi(G)$, where $\chi(G)$ denotes the chromatic number of $G$.
     
	Several authors have investigated the local distance antimagic chromatic numbers for various classes of graphs, see  \cite{yamini,Handa,Nalliah1,Nalliah2,Nalliah3,Nalliah4} for further results and related studies. Handa et al.~\cite{Handa} established the following result, which provides a useful lower bound for the local distance antimagic chromatic number of a graph.
	\begin{proposition}{\rm\cite{Handa}}\label{handaprop}
		Let $G$ be a local distance antimagic graph. If $u$ and $v$ are vertices such that $|N(u)\triangle N(v)|=1$  or $2$, then $w(u)\not=w(v)$, where $\triangle$ denotes the symmetric difference between sets.
	\end{proposition} 
	The following result, due to Priyadharshini et al.~\cite{Nalliah1}, provides a lower bound for the local distance antimagic chromatic number of trees in terms of the number of their support vertices, that is, vertices adjacent to at least one leaf.
	\begin{theorem}{\rm\cite{Nalliah1}}\label{nalliahleaf}
		Let $T$ be a tree on $n\geq 3 $ vertices such that the set of distinct neighbors of all the leaves is of cardinality $t$. Then $\chi_{ld}(T)\geq t+1$.
	\end{theorem}

	Next, we present several classes of graphs that admit a local distance antimagic labeling and have already been studied in the literature. 
	\begin{theorem}\label{th:cycle}{\rm\cite{Handa,Nalliah4}}\label{cycle}
    The cycle $C_n$ on $n$ vertices is local distance antimagic with
		\begin{align*}\chi_{ld}(C_n)&=
			\begin{cases}
				2 &\text{if $n=4$},\\
				3 &\text{if $n \in \{3,12\}$},\\
				4 &\text{if $n \in \{6,8,10,14\}$},\\
				5 &\text{if $n \in \{5,7,9\}$},
			\end{cases}\\
			4\leq \chi_{ld}(C_n)&\leq 5\quad\text{if $n\in \{11,13\}$},\\
			4\leq \chi_{ld}(C_n)&\leq 6\quad\text{if $n\geq 15$}.
		\end{align*}
	\end{theorem}
	\begin{theorem}{\rm\cite{Handa,Nalliah4}}\label{path}
    The path $P_n$ on $n$ vertices is local distance antimagic with
		\begin{align*}
			\chi_{ld}(P_n)&=
			\begin{cases}
				2 &\text{if $n \in \{2,3\}$},\\
				3 &\text{if $n \in \{5,11\}$},\\
				4 &\text{if $n \in \{4,6,7,8,9,10\}$},
			\end{cases}\\
			4 \leq \chi_{ld}(P_n)&\leq 
			\begin{cases}
				5 &\text{if $n\geq 12$, $n$ is even},\\
				6 &\text{if $n\geq 13$, $n$ is odd}.
			\end{cases}
		\end{align*}
	\end{theorem}
	\begin{theorem}{\rm\cite{Handa}}\label{complete multipartite}
		The complete multipartite graph $G=K_{n_1,n_2,\dots,n_r}$
		is local distance antimagic
		with $\chi_{ld}(G) = r$.
	\end{theorem}
	\begin{theorem}{\rm\cite{Handa}}
		The wheel $W_n$ on $n+1$ vertices, $n\geq 3$, is local distance antimagic with $3\leq \chi_{ld}(W_n) \leq 7$.
	\end{theorem}
	\begin{theorem}{\rm\cite{Handa}}
		The complete graph $K_n$, $n\geq 2$, is local distance antimagic with $\chi_{ld}(K_n)=n$.
	\end{theorem}

	\begin{theorem}{\rm\cite{yamini,Handa}}\label{friendship}
		The friendship graph $F_n$, $n\geq 2$, is local distance antimagic with $\chi_{ld}(F_n)=2n+1$.
	\end{theorem}
	
	\begin{theorem}{\rm\cite{yamini}}\label{bistar}
		The bistar graph $B_{m,n}$, $m,n\geq 2$, is local distance antimagic with $\chi_{ld}(B_{m,n})=4.$
	\end{theorem} 
In our investigation, we also make use of the concept of a magic rectangle. Let $a, b$ be positive integers.
A {\it magic rectangle} $MR(a,b)=(m_{i,j})$ is an $a\times b$ array containing the integers $1,2,\dots, ab$ each appearing exactly once, such that all row sums are equal to a constant $\rho$ and all column sums are equal to a constant $\sigma$. The sum of all entries in $MR(a,b)$ is ${ab(ab+1)}/{2}$, and the corresponding magic constants are given by 
\begin{align*}
    \sigma=\sum_{i=1}^am_{ij}=\tfrac{a(ab+1)}{2}, &\quad\text{for any $j\in \{1,2,\dots,b\}$ and }\\
     \rho=\sum_{j=1}^bm_{ij}=\tfrac{b(ab+1)}{2}, &\quad\text{for any $i\in \{1,2,\dots,a\}$.}
    \end{align*}
    Harmuth \cite{MR1, MR2} proved that such arrays exist whenever $a$ and $b$ are of the same parity, except when exactly one of $a$ and $b$ is $1$, or when $a=b=2$.
    \begin{theorem}{\rm\cite{MR1,MR2}}\label{magicrectangle}
        A magic rectangle $MR(a,b)$ exists if and only if $a,b>1$, $ab>4$ and $a\equiv b\pmod 2$.
    \end{theorem}
    
Throughout this paper, all graphs 
 $G$ are assumed to admit a local distance antimagic labeling, with local distance antimagic chromatic number $\chi_{ld}(G)$. For a graph $G$, and a vertex $v\in V(G)$, we denote by $N_G(v)$ the neighborhood of   $v$ in $G$ and  by $w_G(v)$ the weight of  $v$. When the context is clear, the subscript $G$ will be omitted from both notations.

Although local distance antimagic labeling and neighborhood balanced coloring are defined independently, the latter provides the structural tool used in our labeling constructions. In a neighborhood balanced \(k\)-coloring, every vertex has the same number of neighbors in each color class. This property enables us to distribute blocks of labels among the color classes so that the additional terms introduced in the direct product or in different copies of a graph cancel when vertex weights are computed. Consequently, the weight of a vertex in the resulting graph can be expressed in terms of its weight under the original local distance antimagic labeling. The concepts developed in Sections~\ref{nbc} and~\ref{nnbc} are therefore essential to the constructions and bounds established in Section~\ref{sec:mainresults}.

 In Section \ref{nbc}, we recall the concept of neighborhood balanced $k$-coloring of graphs. Motivated by this concept, Section \ref{nnbc} introduces the notion of nearly neighborhood balanced $k$-coloring of graphs and presents several illustrative examples. Furthermore, at the beginning of Section \ref{sec:mainresults}, we provide a counterexample to the characterization of graphs with local distance antimagic chromatic number two, claimed by Priyadharshini et al. \cite{Nalliah1}. In Subsections \ref{sec:direct} and \ref{sec:union}, we use the concepts of neighborhood balanced $k$-coloring and nearly neighborhood balanced $k$-coloring to study the local distance antimagic labeling of direct products and unions of graphs.

     
	\section{Neighborhood Balanced \texorpdfstring{$k$-}CColoring of Graphs}\label{nbc}
	Freyberg et al.~\cite{NBC} introduced the concept of \emph{neighborhood balanced coloring}.
A neighborhood balanced coloring of a graph $G$ is a vertex coloring of $G$ using two colors, say red and green, such that each vertex has the same number of neighbors in both   color classes.
A graph that admits a neighborhood balanced coloring is called \emph{neighborhood balanced colorable} or a \emph{neighborhood balanced colored graph} (see Figure \ref{Fig:nbc}). The authors investigated structural properties of graphs that admit such a coloring and presented several classes of graphs possessing this property.

Minyard et al.~\cite{3nbc} extended this notion and introduced the concept of  neighborhood balanced $3$-coloring.
A graph is said to admit a {\it neighborhood balanced $3$-coloring} if its vertices can be colored with three colors such that every vertex has an equal number of neighbors in each color class.
They also characterized several families of graphs that admit such a coloring.

Almeida et al.~\cite{k-nbc} further generalized the concepts to $k\ge 2$ colors as follows. If the vertices of a graph $G$ can be colored using the $k$ colors, say $1, 2, \dots, k$, so that every vertex has an equal number of neighbors in each color class, then the coloring is called a \emph{\nbc}.
A graph that admits such a coloring is referred to as \emph{\nbcl} or a \emph{\nbcd graph}.

Equivalently, a \nbc of graph $G$ is a partition of its vertex set $V$ of $G$ into $k$ subsets $V_1, V_2, \dots, V_k$ (that is, the color classes) such that every vertex has the same number of neighbors in each subset. Let $|E(V_i,V_j)|$ denote the number edges between the color classes $V_i$ and $V_j$, and  let $|V_i|$ denote the number of vertices in   $V_i$.

It is often convenient to represent a $k$-coloring by a function  $c\colon V(G)\to \{1,2,\dots, k\}$. 
Define $w(v)=\sum_{u\in N(v)}c(u)$.
If $c$ is a neighborhood balanced $k$-coloring of $G$, then
$w(v)\equiv 0 \pmod{ {k(k+1)}/{2}}$
for every $v\in V(G)$.
 For further details, see \cite{k-nbc}.
 
 The following lemma provides a necessary condition for a graph to admit a \nbc.
 \begin{lemma}\label{th:prep1}{\rm\cite{k-nbc}}
If a graph $G$ admits a \nbc, then the degree of every vertex is a multiple of $k$. 
\end{lemma}
The following lemma indicates that a neighborhood balanced coloring with $k$ colors can be reduced to one with fewer colors. In particular, if $k$ is even, the graph admits a neighborhood balanced $2$-coloring, and if $k$ is odd, it admits a neighborhood balanced $p$-coloring, where 
$p$ is the smallest prime divisor of $k$.
\begin{lemma}{\rm\cite{k-nbc}}\label{reduction}
    If a graph $G$ admits \nbc and $p$ divides $k$, then the graph $G$ admits neighborhood balanced $p$-coloring.
\end{lemma}
The converse of this lemma does not hold in general.
For instance, the cycle $C_m$  on $m$ vertices  admits a neighborhood balanced $2$-coloring whenever $m\equiv 0 \pmod 4$, but it does not admit a neighborhood balanced $k$-coloring for any even $k>2$.

The following three results describe further structural properties of graphs that admit a \nbc.
	\begin{theorem}{\rm\cite{k-nbc}}
		 Suppose $G$ admits a \nbc. Then for any \nbc $c$,
    $$|E(V_i, V_j)|=\tfrac{2|E(G)|}{k^2}  \text{\ and  \ }|E(V_i, V_i)|=\tfrac{|E(G)|}{k^2}. $$
	\end{theorem}
	\begin{cor}\label{th: regular graphs necessary condition} {\rm\cite{k-nbc}}
    Suppose $G$ is an $r$-regular graph that admits a \nbc. Then $|V_i|={|V(G)|}/{k}$.
\end{cor}
\begin{cor}\label{cor2}{\rm\cite{k-nbc}}
    If $G$ is an $r$-regular graph that admits a \nbc, then $|V(G)| \equiv 0\pmod k$ and $|E(G)| \equiv 0\pmod k$.
\end{cor}
In the following, we discuss several families of graphs that have been examined in the literature concerning the existence of neighborhood balanced   $k$-coloring.
	\begin{theorem}{\rm\cite{NBC}}
		The cycle $C_n$ admits a neighborhood balanced $2$-coloring if and only if $n\equiv 0\ (\bmod\ 4)$.
	\end{theorem}
	
	\begin{theorem}{\rm\cite{k-nbc}}\label{nbc-multipartite}
		Let $p \geq 2$. The complete multipartite graph $K_{n_1,n_2,\dots,n_p}$ admits a \nbc if and only if $n_i\equiv 0\pmod k$ for $i=1,2, \dots, p$.
	\end{theorem}
    \begin{theorem}{\rm\cite{k-nbc}} \label{th: hamming graphs}
The hamming graph $H(d,k)$ is a \nbcd graph if and only if $d\equiv 0 \pmod k$.
\end{theorem}

\begin{figure}[ht!]
    \centering
    		\begin{tikzpicture}[auto]
\tikzstyle{vertex}=[fill, circle, minimum size=6pt, inner sep=0pt]
\tikzstyle{edge}=[draw=black,line width=0.3mm]

\node[circle, inner sep = 3pt,, fill=red] (v0) at (90+0*360/8:1.5cm) {};
\node[circle, inner sep = 3pt,, fill=red] (v1) at (90+1*360/8:1.5cm) {};
\node[circle, inner sep = 3pt,, fill=green] (v2) at (90+2*360/8:1.5cm) {};
\node[circle, inner sep = 3pt,, fill=green] (v3) at (90+3*360/8:1.5cm) {};
\node[circle, inner sep = 3pt,, fill=red] (v4) at (90+4*360/8:1.5cm) {};
\node[circle, inner sep = 3pt,, fill=red] (v5) at (90+5*360/8:1.5cm) {};
\node[circle, inner sep = 3pt,, fill=green] (v6) at (90+6*360/8:1.5cm) {};
\node[circle, inner sep = 3pt,, fill=green] (v7) at (90+7*360/8:1.5cm) {};

\node   at (90+0*360/8:1.9cm) {1};
\node   at (90+1*360/8:1.9cm) {1};
\node   at (90+2*360/8:1.9cm) {2};
\node   at (90+3*360/8:1.9cm) {2};
\node   at (90+4*360/8:1.9cm) {1};
\node   at (90+5*360/8:1.9cm) {1};
\node   at (90+6*360/8:1.9cm) {2};
\node   at (90+7*360/8:1.9cm) {2};

\draw[edge] (v0) -- (v1)-- (v2)-- (v3)-- (v4)-- (v5)-- (v6)-- (v7)--  (v0);


\node  at ($(90+0*360/8:1.5cm) + (4cm,0.4cm)$) {1};
\node  at ($(90+0*360/8:1.5cm) + (5cm,0.4cm)$) {1};
\node  at ($(90+0*360/8:1.5cm) + (6cm,0.4cm)$) {1};
\node  at ($(90+0*360/8:1.5cm) + (7cm,0.4cm)$) {2};
\node  at ($(90+0*360/8:1.5cm) + (8cm,0.4cm)$) {2};
\node  at ($(90+0*360/8:1.5cm) + (9cm,0.4cm)$) {2};
\node  at ($(90+4*360/8:1.5cm) + (5cm,-0.4cm)$) {1};
\node  at ($(90+4*360/8:1.5cm) + (6cm,-0.4cm)$) {1};
\node  at ($(90+4*360/8:1.5cm) + (7cm,-0.4cm)$) {2};
\node  at ($(90+4*360/8:1.5cm) + (8cm,-0.4cm)$) {2};

\node[circle, inner sep = 3pt,, fill=red] (v10) at ($(90+0*360/8:1.5cm) + (4cm,0cm)$) {};
\node[circle, inner sep = 3pt,, fill=red] (v11) at ($(90+0*360/8:1.5cm) + (5cm,0cm)$) {};
\node[circle, inner sep = 3pt,, fill=red] (v12) at ($(90+0*360/8:1.5cm) + (6cm,0cm)$) {};
\node[circle, inner sep = 3pt,, fill=green] (v13) at ($(90+0*360/8:1.5cm) + (7cm,0cm)$) {};
\node[circle, inner sep = 3pt,, fill=green] (v14) at ($(90+0*360/8:1.5cm) + (8cm,0cm)$) {};
\node[circle, inner sep = 3pt,, fill=green] (v15) at ($(90+0*360/8:1.5cm) + (9cm,0cm)$) {};

\node[circle, inner sep = 3pt,, fill=red] (v21) at ($(90+4*360/8:1.5cm) + (5cm,0cm)$) {};
\node[circle, inner sep = 3pt,, fill=red] (v22) at ($(90+4*360/8:1.5cm) + (6cm,0cm)$) {};
\node[circle, inner sep = 3pt,, fill=green] (v23) at ($(90+4*360/8:1.5cm) + (7cm,0cm)$) {};
\node[circle, inner sep = 3pt,, fill=green] (v24) at ($(90+4*360/8:1.5cm) + (8cm,0cm)$) {};

\draw[edge] (v10) -- (v21);
\draw[edge] (v10) -- (v22);
\draw[edge] (v10) -- (v23);
\draw[edge] (v10) -- (v24);

\draw[edge] (v11) -- (v21);
\draw[edge] (v11) -- (v22);
\draw[edge] (v11) -- (v23);
\draw[edge] (v11) -- (v24);

\draw[edge] (v12) -- (v21);
\draw[edge] (v12) -- (v22);
\draw[edge] (v12) -- (v23);
\draw[edge] (v12) -- (v24);

\draw[edge] (v13) -- (v21);
\draw[edge] (v13) -- (v22);
\draw[edge] (v13) -- (v23);
\draw[edge] (v13) -- (v24);

\draw[edge] (v14) -- (v21);
\draw[edge] (v14) -- (v22);
\draw[edge] (v14) -- (v23);
\draw[edge] (v14) -- (v24);

\draw[edge] (v15) -- (v21);
\draw[edge] (v15) -- (v22);
\draw[edge] (v15) -- (v23);
\draw[edge] (v15) -- (v24);
\end{tikzpicture}

 \caption{Neighborhood balanced $2$-coloring of $C_8$ and $K_{4,6}$.}\label{Fig:nbc}
\end{figure}

\section{Nearly Neighborhood Balanced \texorpdfstring{$k$-}CColoring of Graphs}\label{nnbc}
Consider the star $K_{1,kn}$ and color its $kn$ pendant vertices so that  each of the $k$ colors is used exactly $n$ times. In this coloring, the central (non-pendant) vertex has $n$ neighbors of each color, however, it  must itself also be assigned one of the $k$ colors. Consequently, the pendant vertices deprive $K_{1,kn}$ from admitting a \nbc. Motivated by this observation, we introduce a concept of a \nnbc of graphs as follows.
\begin{dfn}
A vertex $k$-coloring of a graph $G$ is called a  \nnbc if every non-pendant vertex of $G$ has the same number of neighbors in each color class.
\end{dfn}
A graph that admits \nnbc is called \emph{\nnbcl} or \emph{\nnbcd}.
We begin by presenting a necessary condition for a graph to admit a \nnbc, which is analogous to the necessary condition for \nbcd graphs.
\begin{lemma}\label{nnbc-deg}
    If a graph $G$ admits a \nnbc, then the degree of every non-pendant vertex of $G$ is a multiple of $k$.
\end{lemma}
The next Lemma is also analogous to Lemma \ref{reduction} for \nbcd graphs.
\begin{lemma}\label{nnbc-reduction}
    If a graph $G$ admits a \nnbc and $p$ divides $k$, then     $G$ admits a nearly neighborhood balanced $p$-coloring.
\end{lemma}
Next, we present  examples of graphs that admit a \nnbc. As described at the beginning of this section, the star $K_{1,kn}$ admits a \nnbc. We now state this observation formally.
\begin{theorem}
    The star $K_{1,n}$ admits a \nnbc if and only if $n\equiv 0 \pmod{k}$.
\end{theorem}

The bistar $B_{m,n}$ is a tree with vertex set $V(B_{m,n})=\{a,b\}\cup \{u_i: 1\leq i \leq m\}\cup\{v_i: 1\leq i \leq n\}$ and edge set $E(B_{m,n})=\{ab\}\cup \{au_i: 1\leq i \leq m\}\cup \{bv_i: 1\leq i \leq n\}$. We now give conditions on $m,n$ for   the bistar $B_{m,n}$ to admit a \nnbc.
\begin{theorem}
    The bistar $B_{m,n}$ admits a \nnbc if and only if $m\equiv -1\pmod{k}$ and  $n\equiv -1 \pmod{k}$. 
\end{theorem}
\begin{proof}
  Suppose that $m\not\equiv -1 \pmod{k}$ or $n \not\equiv -1 \pmod{k}$. Then the degree of   vertex $a$ or vertex  $b$ is not a multiple of $k$. By Lemma~\ref{nnbc-deg},  $B_{m,n}$  therefore does not admit a \nnbc.

  Conversely, assume that $m\equiv -1  \pmod{k}$ and  $n\equiv -1  \pmod{k}$. Color the vertices $u_1,u_2,\dots,$ $u_m$, which are  adjacent to $a$, using each of the $k$ colors in a cyclic order. Since $m \equiv -1  \pmod{k}$, one of the $k$ colors is used exactly once fewer than each of the remaining $k-1$ colors. Assign this color to the vertex $b$. Apply the same strategy to color the neighbors of vertex $b$, noting that $n\equiv -1 \pmod k$. Under this coloring, every non-pendant vertex of $B_{m,n}$  has the same number of neighbors in each color class.
Hence, the resulting coloring is a \nnbc of $B_{m,n}$.
\end{proof}
Next, we examine the nearly neighborhood balanced $2$-coloring of the path $P_n$. Let $V(P_n)=\{v_i:  1\leq i \leq n\}$ and $E(P_n)=\{v_iv_{i+1}: 1\leq i \leq n-1\}$ be the vertex and edge sets of $P_n$, respectively.
\begin{theorem}
    The path $P_n$ is a nearly neighborhood balanced $2$-colored graph for all $n\geq 3$.
\end{theorem}
\begin{proof}
  We define a coloring $c\colon V(P_n)\to\{1,2\}$ as follows;
  \begin{align*}
      c(v_i)=\begin{cases}
          1 &\text{if $i\equiv 1,2 \pmod 4$},\\
          2 &\text{if $i\equiv 0, 3 \pmod 4$}. 
      \end{cases}
  \end{align*}
  By the definition of the coloring $c$, the vertices at distance two receive different colors. 
  Each non-pendant vertex $v_i$ of $P_n$, $i=2, 3, \dots, n-1$, has exactly two neighbors, $v_{i-1}$ and $v_{i+1}$, which are assigned distinct colors.
Therefore, every non-pendant vertex has one neighbor of each color, and hence $c$ is a nearly neighborhood balanced $2$-coloring  of $P_n$.
\end{proof}
Now we present some nearly \nbcd graphs that can be constructed from \nbcd graphs using the operation called the corona product of graphs.
The {\it corona product} of two graphs $G$ and $H$, denoted by $G\circ H$, is a graph obtained by taking one copy of $G$ and $|V(G)|$ copies of $H$, and joining the $i$th vertex of $G$ to every vertex in the $i$th copy of $H$. By definition,  if $G$ is a graph of order $m$ and $H$ is a graph of order $n$, then $G\circ H$ has order $m+mn$. Furthermore, if $V(G)=\{v_i:  1\leq i \leq m\}$ and $V(H)=\{x_j:  1\leq j \leq n\}$ denote the vertex sets of graphs $G$ and $H$, respectively, then  the vertex set of the graph $G\circ H$ is given by $V(G\circ H)=V(G)\cup\bigcup_{i=1}^m\{x_{i,j}:  1\leq j \leq n\}$, where each set $\{x_{i,1}, x_{i,2},\dots,x_{i,n}\}$ corresponds to the vertices of the $i$th copy of $H$ attached to the vertex $v_i\in V(G)$.
\begin{theorem}
    If $G$ is a \nbcd graph, then $G\circ \overline{K_n}$ is a \nnbcd if and only if $n\equiv 0 \pmod k$, where $\overline{K_n}$ denotes the complement of the complete graph on $n$ vertices. 
\end{theorem}
\begin{proof}
Let $c$ be a \nbc of $G$. Since $G$ admits a \nbc, by Lemma \ref{th:prep1}, the degree of each vertex of $G$ is a multiple of $k$.

If $n\not\equiv 0 \pmod k$, then the degree of non-pendant vertices of $G\circ \overline{K_n}$ (all vertices of $G$ are non-pendant vertices of $G\circ \overline{K_n}$)  is not a multiple of $k$ and therefore, by 
Lemma~\ref{nnbc-deg}, $G\circ\overline{K_n}$ does not admit a nearly neighborhood balanced $k$-coloring.

Conversely, suppose $n\equiv 0 \pmod k$. Define a coloring $c': V(G\circ \overline{K_n})\to\{1,2,\dots,k\}$ by 
\begin{align*}
    c'(x)=
    \begin{cases}
        c(x)&\text{if $x\in V(G)$},\\
        p &\text{if $x=x_{i,j}$, $1\leq i \leq m$, and  $j\equiv p \pmod k$,   $1\leq j \leq n$.}
    \end{cases}
\end{align*}
Consider any non-pendant vertex $v_i$ of $G\circ \overline{K_n}$. Among the vertices of $G$, its neighbors  are colored using $c$, so  $v_i$ has an equal number of neighbors in each of the 
$k$ color classes. Moreover, since $n\equiv 0 \pmod k$, the vertices $x_{i,1},x_{i,2},\dots, x_{i,n}$ are colored equally using each of the $k$ colors. Therefore, $v_i$ continues to have the same number of neighbors of each color, satisfying the condition for a nearly neighborhood balanced coloring. Thus, $G\circ \overline{K_n}$ is a \nnbcd.     
\end{proof}
The above result demonstrates that every \nbcd graph is an induced subgraph of a \nnbcd.


\section{Main Results}\label{sec:mainresults}
Priyadharshini et al.~\cite{Nalliah1}  claimed to have classified connected graphs with local distance antimagic chromatic number 2 in the following theorem. 
	\begin{theorem}{\rm\cite{Nalliah1}}\label{child2}
		A connected graph $G$ has $\chi_{ld}(G)=2$ if and only if it is a complete bipartite graph.
	\end{theorem}
However, we have found a counterexample to this result, which shows that the above classification of connected graphs admitting local distance antimagic chromatic number~2 is incorrect. Our counterexample is obtained by taking the direct product of a regular graph with a complete bipartite graph.

For completeness, we recall the definition of the direct product of graphs. The {\it direct product} of graphs $G$ and $H$, denoted by $G\times H$, is the graph with vertex set $V(G)\times V(H)$, where two vertices $(g,h)$ and $(g',h')$ are adjacent if and only if $g$ is adjacent to $g'$ in $G$ and $h$ is adjacent to $h'$ in $H$.
\begin{theorem}\label{directwithbipartite}
Let $n_1, n_2\geq 2$ be positive integers and let $G$ be an $r$-regular graph of order $n$, where $r\ge 1$. Then $\chi_{ld}(G\times K_{n_1,n_2})=2$ if $n_1$ and $n_2$ have the same parity, or if they have opposite parity and $n$ is odd. 
	\end{theorem}
	\begin{proof}
		Without loss of generality, assume $n_1\leq n_2$. Let $V(G)=\{x_1,x_2,\dots,x_n\}$ and $V(K_{n_1,n_2})=\{y_1,y_2,\dots,y_{n_1}\}\cup\{z_1,z_2,\dots,z_{n_2}\}$ be the vertex sets of graphs $G$ and $K_{n_1,n_2}$, respectively. Further, let the vertex set of graph $G\times K_{n_1,n_2}$ be $V(G\times K_{n_1,n_2})=\{(x_i,y_j): 1\leq i \leq n, 1\leq j \leq n_{1}\}\cup \{(x_i,z_j):\ 1\leq i \leq n, 1\leq j \leq n_{2}\}$. Throughout the proof, we   assume that the index $i$ takes values from $1,2,\dots,n$. Our labeling will be constructed based on the parities of $n$, $n_1$, and $n_2$.
		\begin{description}
			\item[ Case 1.]  $n_1$ and $n_2$ are both even.

			Define a bijection $f\colon V(G\times K_{n_1,n_2})\to \{1,2,\dots,(n_1+n_2)n\}$ by
			\begin{align*}
				f(x_i,y_j)&=
				\begin{cases}
					(j-1)n+i &\text{if $j=1,3,5,\dots,n_1-1$,}\\
					jn+1-i &\text{if $j=2,4,6,\dots,n_1$,}
				\end{cases}\\
				f(x_i,z_j)&=
				\begin{cases}
					(j-1)n+i+nn_1 &\text{if $j=1,3,5,\dots,n_2-1$,}\\
					jn+1-i+nn_1 &\text{if $j=2,4,6,\dots,n_2$.}
				\end{cases}
			\end{align*}
Note that for fixed $i$
\begin{align*}
\sum_{j=1}^{n_2}f(x_i,z_j)&=\sum_{j=1}^{\tfrac{n_2}{2}}\left((2j-2)n+i+nn_1+2jn+1-i+nn_1\right)=\sum_{j=1}^{\tfrac{n_2}{2}}4jn-nn_2+\tfrac{n_2}{2}+nn_1n_2\\
&=nn_2\left(\tfrac{n_2}{2}+1\right)-nn_2+\tfrac{n_2}{2}+nn_1n_2=\tfrac{n_2(nn_2+1)}{2}+nn_1n_2,
\end{align*}
and similarly
\begin{align*}
\sum_{j=1}^{n_1}f(x_i,y_j)&=\sum_{j=1}^{\tfrac{n_1}{2}}\left((2j-2)n+i+2jn+1-i\right)
=\tfrac{n(nn_1+1)}{2}.
\end{align*}
            
			Using the above sums, the weight of the vertices is given by
			\begin{align*}
				w(x_i,y_j)&=\sum_{x_q\in N_G(x_i)}\sum_{p=1}^{n_2}f(x_q,z_p)=\sum_{x_q\in N_G(x_i)}\left(\tfrac{n_2(nn_2+1)}{2}+nn_1n_2\right)
                =\tfrac{rn_2(nn_2+1)}{2}+rnn_1n_2
                \end{align*}
                for $1\leq j \leq n_1$, and
                \begin{align*}
				w(x_i,z_j)&=\sum_{x_q\in N_G(x_i)}\sum_{p=1}^{n_1}f(x_q,y_p) =\sum_{x_q\in N_G(x_i)}\tfrac{n(nn_1+1)}{2}
                 =\tfrac{rn_1(nn_1+1)}{2}
			\end{align*}
            for $1\leq j \leq n_2$.
            Clearly, the weights of adjacent vertices are distinct. Therefore, $f$ is a local distance antimagic labeling of $G\times K_{n_1,n_2}$ that produces exactly two distinct vertex weights.
			\item[ Case 2.]  $n_1$ and $n_2$ are both odd.
            
			Define a bijection $f\colon V(G\times K_{n_1,n_2})\to \{1,2,\dots,(n_1+n_2)n\}$ by
			\begin{align*}
				f(x_i,y_j)&=
				\begin{cases}
					n(j-1)+i &\text{if $j=1,2$,}\\
					4n-2(i-1) &\text{if $j=3$,}\\
                    nj+i &\text{if $j=4,6,\dots,n_1-1$, }\\
					n(j+1)-(i-1) &\text{if $j=5,7,\dots,n_1$,}
				\end{cases}\\
				f(x_i,z_j)&=
				\begin{cases}
					4n-(2i-1) &\text{if $j=1$, }\\
					n(n_1+2)+i&\text{if $j=3$, }\\
                    n(n_1-1+j)+i &\text{if $j=2,4,\dots,n_2-1$,}\\
					n(n_1+j)-(i-1) &\text{if $j=5,7,\dots,n_2$.}
				\end{cases}
			\end{align*}
Note that for fixed $i$
\begin{align*}
 \sum_{j=1}^{n_1}f(x_i,y_j)&=5n+2+\sum_{j=4}^{n_1}f(x_i,y_j)=5n+2+\sum_{p=2}^{\tfrac{n_1-1}{2}}\left(2np+i+2n(p+1)-i+1\right)\\&=5n+2+\sum_{p=2}^{\tfrac{n_1-1}{2}}\left(4np+2n+1\right)=5n+2+(2n+1)\tfrac{(n_1-3)}{2}+\tfrac{4n(n_1^2-9)}{8}\\& =5n+2+ \tfrac{1}{2}\left(nn_1(n_1+2)+n_1-3-15n\right)=\tfrac{1}{2}\left(nn_1(n_1+2)-5n+n_1+1\right),
  \end{align*}
                and similarly
                \begin{align*}
\sum_{j=1}^{n_2}f(x_i,z_j)&=7n+2nn_1+1+\sum_{j=4}^{n_2}f(x_i,z_j)\\
 &=7n+2nn_1+1+\sum_{p=2}^{\tfrac{n_2-1}{2}}\left(n(n_1+2p-1)+i+n(n_1+2p+1)-i+1\right)\\
 &=7n+2nn_1+1+\sum_{p=2}^{\tfrac{n_2-1}{2}}\left(4np+2nn_1+1\right)\\
 &=7n+2nn_1+1+(2nn_1+1)\tfrac{(n_2-3)}{2}+\tfrac{4n(n_2^2-9)}{8}\\
 &=7n+2nn_1+1+ \tfrac{1}{2}\left(2nn_1n_2-6nn_1+n_2-3+nn_2^2-9n\right)\\
 &=\tfrac{1}{2}\left(2nn_1(n_2-1)+n_2(nn_2+1)+5n-1\right).
\end{align*}
			For the weight of vertices, we have,
			\begin{align*}
				w(x_i,y_j)&=\sum_{x_q\in N_G(x_i)}\sum_{p=1}^{n_2}f(x_q,z_p)=\sum_{x_q\in N_G(x_i)}\left(\tfrac{1}{2}\left(2nn_1(n_2-1)+n_2(nn_2+1)+5n-1\right)\right)
                 \\
                &=\tfrac{r}{2}\left(2nn_1(n_2-1)+n_2(nn_2+1)+5n-1\right)\end{align*}
                for $1\leq j \leq n_1$, and
                \begin{align*}
				 w(x_i,z_j)&=\sum_{x_q\in N_G(x_i)}\sum_{p=1}^{n_1}f(x_p,y_q) =\sum_{x_q\in N_G(x_i)}\left(           \tfrac{1}{2}\left(nn_1(n_1+2)-5n+n_1+1 \right)\right) 
                \\&=\tfrac{r}{2}\left(nn_1(n_1+2)-5n+n_1+1 \right) 
			\end{align*}
            for $1\leq j \leq n_2$.
            Thus, the   weights of adjacent vertices are distinct in this case as well. Hence, $f$ is a local distance antimagic labeling of $G\times K_{n_1,n_2}$ that produces  two distinct vertex weights.
            
			\item[ Case 3.]  $n_1$ and $n_2$ have  opposite parity and $n$ is odd.
            
            Without loss of generality, assume that $n_1$ is odd and $n_2$ is even.
			Since $n\equiv n_1\pmod 2$, it follows from  Theorem~\ref{magicrectangle} that there exists a magic rectangle $A$ of size $n\times n_1$ with entries $a_{i,j}$ having constant row sums equal to ${n_1(nn_1+1)}/{2}$.
			Define a bijection $f\colon V(G\times K_{n_1,n_2})\to \{1,2,\dots,(n_1+n_2)n\}$ by  
			\begin{align*}
            f(x_i,y_j)&=a_{i,j} \quad\text{if $j=1, 2, \dots, n_1$},\\
				f(x_i,z_j)&=
				\begin{cases}
					(j-1)n+i+nn_1 &\text{if $j=1,3,5,\dots,n_2-1$,}\\
					jn+1-i+nn_1 &\text{if $j=2,4,6,\dots,n_2$.}
				\end{cases}
			\end{align*}
As seen in Case 1, $\sum_{j=1}^{n_2}f(x_i,z_j)={n_2(nn_2+1)}/{2}+nn_1n_2$, and as $\sum_{j=1}^{n_1}f(x_i,y_j)$ equals to the row sum of the magic rectangle $A$, we have 
	$\sum_{j=1}^{n_1}f(x_i,y_j)= {n_1(nn_1+1)}/{2}$.
            Using these sums, the weight of the vertices is given by 
			\begin{align*}
				w(x_i,y_j)&=\sum_{x_q\in N_G(x_i)}\sum_{p=1}^{n_2}f(x_q,z_p)=\tfrac{rn_2(nn_2+1)}{2}+rnn_1n_2\end{align*}
                for $1\leq j \leq n_1$, and
                \begin{align*}
				w(x_i,z_j)&=\sum_{x_q\in N_G(x_i)}\sum_{p=1}^{n_1}f(x_q,y_p) =\tfrac{rn_1(nn_1+1)}{2}
			\end{align*}
            for $1\leq j \leq n_2$.
            Again, we get that the adjacent  vertices have distinct weights. Therefore, $f$ is a local distance antimagic labeling of $G\times K_{n_1,n_2}$ that assigns two distinct weights to its vertices.
        \end{description}\vspace*{-8mm}\end{proof}
Note that the assumptions on  parity in  Theorem \ref{directwithbipartite} are necessary.  If $n$ is even and $n_1$ and $n_2$ have opposite parity, then $\chi_{ld}(G\times K_{n_1,n_2})$  does not necessarily equal 2.    Priyadharshini et al.~\cite{Nalliah1} proved that $\chi_{ld}(2K_{2,3})=3$ and since $K_2\times K_{2,3}\cong 2 K_{2,3}$, we obtain $\chi_{ld}(K_2\times K_{2,3})=3$.

Weichsel et al.~\cite{dirconnected} proved that the direct product of two graphs $G$ and $H$ is connected if and only if both $G$ and $H$ are connected and at least one of them is non-bipartite. This implies that the direct product of an odd cycle $C_{2n+1}$ with $K_{n_1,n_2}$ is connected. According to Theorem~\ref{directwithbipartite},  we have $\chi_{ld}(C_{2n+1}\times K_{n_1,n_2})=2$. However, one can easily verify that the graph $C_{2n+1}\times K_{n_1,n_2}$ is not a complete bipartite graph. This provides a counterexample to  Theorem \ref{child2}.\par \bigskip

We now make an attempt to characterize graphs whose local distance antimagic chromatic number equal to 2. We begin with the following trivial observation.
\begin{obs}
    If $G$ is a graph with $\chi_{ld}(G)=2$, then $G$ is bipartite.
\end{obs}

\begin{theorem}\label{bipartitechild}
    Let $G$ be a connected bipartite graph of order $n$. If there exists a pair of vertices $u$ and $v$ belonging to same partite set satisfying one of the following conditions: 
    \begin{enumerate}
        \item (i) $N(u)\subset N(v)$ or (ii) $N(v)\subset N(u)$,
        \item $|N(u)- N(v)|=a\geq 1$ and $|N(v)- N(u)|=b\geq 1$, such that for all mutually exclusive non empty proper subsets $A,B$ of $\{1,2,\dots,n\}$ with $|A|=a$, $|B|=b$, the sum of element(s) in $A$ and sum of element(s) in $B$ are distinct,
    \end{enumerate}
     then $\chi_{ld}(G)>2$.
\end{theorem}
\begin{proof}
    Let $f$ be a local distance antimagic labeling of $G$. Let $u$ and $v$ be vertices belonging to one partite set satisfying the given condition.
    \begin{enumerate}
         \item (i) If $N(u)\subset N(v)$, then, $w(u)<w(v)$. (ii) By the same argument, we have $w(v)<w(u)$.
        \item By the given hypothesis, under each possible labeling $f$, we must have $w(u)\neq w(v).$
    \end{enumerate}
    So we get at least two distinct weights among the vertices belonging to the partite set containing $u$ and $v$. Further, the weight of any vertex in the other partite set should be distinct from these two weights. Thus, $\chi_{ld}(G)>2$.
\end{proof}
\begin{cor}
   Let $G$ be a bipartite graph of order $n$. If there exists a pair of vertices $u$ and $v$ belonging to a partite set such that $|N(u)\triangle N(v)|=1\ \text{or}\ 2$, where $\triangle$ denotes the symmetric difference between sets, then $\chi_{ld}(G)>2$.  
\end{cor}

Theorem~\ref{bipartitechild} provides another class of graphs whose local distance antimagic chromatic number is greater than \(2\). Nevertheless, a complete characterization of graphs with local distance antimagic chromatic number \(2\) remains a fundamental open problem.
\begin{problem}
Characterize graphs \(G\) such that \(\chi_{ld}(G)=2\).
\end{problem}


\subsection{Local Distance Antimagic Labeling of Direct Product of Graphs}\label{sec:direct}
The direct product of graphs has attracted significant attention from the research community. One of the most well-known problems related to this operation is Hedetniemi's conjecture asserted that  $$\chi(G\times H)=\min \{\chi(G),\chi(H)\}.$$
This conjecture suggests that the chromatic number of the direct product of two graphs depends only on the smaller chromatic number of the factors. Hedetniemi's conjecture is, however, disproved by Shitov \cite{Shitov} in 2019.

In Theorem \ref{directwithbipartite}, we observed that the local distance antimagic chromatic number of the direct product of a regular graph with certain complete bipartite graphs equals 2, which is precisely the minimum of the local distance antimagic chromatic numbers of the two factors.
However, this does not hold true in general. We provide an example showing that the local distance antimagic chromatic number of the direct product can exceed the local distance antimagic chromatic numbers of both constituent graphs.

\begin{theorem}\label{th:nodirect}
For a positive integer $n\geq 2$,
\begin{align*}
    n+1\leq \chi_{ld}(K_2\times K_n)\leq
    \begin{cases}
        n+2, & \text{if $n$ is even,}\\
        n+1, & \text{if $n$ is odd.}
    \end{cases}
\end{align*}
\end{theorem}

\begin{proof}
Let $V(K_2)=\{x,y\}$, $V(K_n)=\{v_i\ :\, 1\leq i \leq n\}$ and  $V(K_2\times K_n)
=
\{(x,v_i),(y,v_i):1\leq i\leq n\}$
be the vertex set of $K_2$, $K_n$ and $K_2\times K_n$ respectively, and let $f$ be an arbitrary
local distance antimagic labeling of $K_2\times K_n$.

For $1\leq i\leq n$,
$
w(x,v_i)
=
\sum_{\substack{k=1\\k\neq i}}^n f(y,v_k).
$
Hence, for $i\neq j$, $w(x,v_i)-w(x,v_j)
=
f(y,v_j)-f(y,v_i)\neq 0,$
since $f$ is a bijection. Thus, $w(x,v_i)\neq w(x,v_j)$
for every $i\neq j$. Similarly,
$w(y,v_i)\neq w(y,v_j)$
for every $i\neq j$.

Moreover, for $i\neq j$, the vertices $(x,v_i)$ and $(y,v_j)$ are
adjacent in $K_2\times K_n$. Hence, the only possible repeated weights
are of the form $w(x,v_i)=w(y,v_i).$
Therefore, the labeling $f$ induces at least $n$ distinct weights.

Suppose that $f$ induces exactly $n$ distinct weights. Then
\[
w(x,v_i)=w(y,v_i)=a_i,
\qquad 1\leq i\leq n.
\]
In particular,
\begin{align*}
a_1
&=
f(y,v_2)+f(y,v_3)+\cdots+f(y,v_n)=
f(x,v_2)+f(x,v_3)+\cdots+f(x,v_n),
\end{align*}
and
\begin{align*}
a_2
&=
f(y,v_1)+f(y,v_3)+\cdots+f(y,v_n)=
f(x,v_1)+f(x,v_3)+\cdots+f(x,v_n).
\end{align*}
Consequently,
\[
a_2-a_1
=
f(y,v_1)-f(y,v_2)
=
f(x,v_1)-f(x,v_2).
\]
Thus, for some constant $\alpha$,
\[
f(x,v_1)-f(y,v_1)
=
f(x,v_2)-f(y,v_2)
=
\alpha.
\]
Repeating the same argument gives $f(x,v_i)-f(y,v_i)=\alpha,
\ 1\leq i\leq n.$

Therefore,
\begin{align*}
n\alpha
&=
\sum_{i=1}^n
\bigl(f(x,v_i)-f(y,v_i)\bigr) =
f(x,v_1)-f(y,v_1)
+
w(x,v_1)-w(y,v_1)=
\alpha.
\end{align*}
Hence, $(n-1)\alpha=0,$
and therefore $\alpha=0$. This implies $f(x,v_i)=f(y,v_i)$
for every $i$, contradicting the bijectivity of $f$.

Thus, every local distance antimagic labeling of $K_2\times K_n$
induces at least $n+1$ distinct weights. Hence,
\[
\chi_{ld}(K_2\times K_n)\geq n+1.
\]

We now establish the upper bound. We distinguish two cases according to the parity of $n$.

\begin{description}

\item[Case 1:]$n$ is odd.\\
Define $f\colon V(K_2\times K_n)\to\{1,2,\dots,2n\}$
by
\begin{align*}
f(x,v_i)
&=
\begin{cases}
n-1 & \text{if $i=1$,}\\
2i-2 & \text{if $2\leq i\leq \tfrac{n-1}{2}$,}\\
2i-1 & \text{if $\tfrac{n+1}{2}\leq i\leq n$,}
\end{cases}
\end{align*}
and
\begin{align*}
f(y,v_i)
&=
\begin{cases}
1 & \text{if $i=1$,}\\
2i-1 & \text{if $2\leq i\leq \tfrac{n-1}{2}$,}\\
2i & \text{if $\tfrac{n+1}{2}\leq i\leq n$.}
\end{cases}
\end{align*}

It can be easily verified that
\[
\sum_{i=1}^n f(x,v_i)
=
\tfrac{2n^2+n-1}{2},
\]
and
\[
\sum_{i=1}^n f(y,v_i)
=
\tfrac{2n^2+n+1}{2}.
\]

For the vertex weights, we have
\begin{align*}
w(x,v_i)
&=
\begin{cases}
\tfrac{2n^2+n-1}{2}
& \text{if $i=1$,}\\[6pt]
\tfrac{2n^2+n+3-4i}{2}
& \text{if $2\leq i\leq\tfrac{n-1}{2}$,}\\[6pt]
\tfrac{2n^2+n+1-4i}{2}
& \text{if $\tfrac{n+1}{2}\leq i\leq n$,}
\end{cases}
\end{align*}
and
\begin{align*}
w(y,v_i)
&=
\begin{cases}
\tfrac{2n^2-n+1}{2}
& \text{if $i=1$,}\\[6pt]
\tfrac{2n^2+n+3-4i}{2}
& \text{if $2\leq i\leq\tfrac{n-1}{2}$,}\\[6pt]
\tfrac{2n^2+n+1-4i}{2}
& \text{if $\tfrac{n+1}{2}\leq i\leq n$.}
\end{cases}
\end{align*}

For $i\neq j$, the above expressions give $w(x,v_i)\neq w(y,v_j).$
Thus, $f$ is a local distance antimagic labeling. Moreover, $w(x,v_i)=w(y,v_i)$,
for  $2\leq i\leq n$.

Therefore, the labeling induces $n-1$ common weights, together with
the two distinct weights $w(x,v_1)$ and 
$w(y,v_1)$.
Hence, the labeling $f$ induces $n+1$
distinct weights, and consequently
$\chi_{ld}(K_2\times K_n)\leq n+1$.

\item[Case 2:] $n$ is even.\\
We consider the cases $n=2,\, 4,\, 6$ separately and then consider $n\geq 8$.

\begin{description}

\item[Subcase 1:] $n=2$.\\
We show the labeling in Figure~\ref{fig:K21}. The numbers inside the
vertices represent the labels, while the numbers outside represent
their corresponding weights. Observe that the labeling induces
$4$ distinct weights. Hence,
$\chi_{ld}(K_2\times K_2)\leq 4=n+2$.

\begin{figure}[!ht]
    \centering
    \begin{tikzpicture}[scale=0.5,
    vertex/.style={
        circle,
        draw,
        minimum size=5mm,
        inner sep=0pt,
        font=\small
    },
    edge/.style={
        thick
    }
]


\node[vertex] (v11) at (0,2) {$1$};
\node[vertex] (v12) at (0,0) {$2$};

\node[vertex] (v21) at (3,2) {$3$};
\node[vertex] (v22) at (3,0) {$4$};


\draw[edge] (v11) -- (v21);
\draw[edge] (v12) -- (v22);


\node[above=3pt of v11] {$3$};
\node[below=3pt of v12] {$4$};
\node[above=3pt of v21] {$1$};
\node[below=3pt of v22] {$2$};

\end{tikzpicture}
    \caption{Local distance antimagic labeling of $K_2\times K_2$.}
    \label{fig:K21}
\end{figure}

\item[Subcase 2:] $n=4$.\\
We show the labeling in Figure~\ref{fig:K22}. The numbers inside the
vertices represent the labels, while the numbers outside represent
their corresponding weights. Observe that the labeling induces
$6$ distinct weights. Hence,
$\chi_{ld}(K_2\times K_4)\leq 6=n+2$.

\begin{figure}[!ht]
    \centering
    \begin{tikzpicture}[scale=0.5,
    vertex/.style={
        circle,
        draw,
        minimum size=5mm,
        inner sep=0pt,
        font=\small
    },
    edge/.style={
        thick
    }
]


\node[vertex] (v11) at (0,3) {$3$};
\node[vertex] (v12) at (0,1) {$4$};
\node[vertex] (v13) at (0,-1) {$5$};
\node[vertex] (v14) at (0,-3) {$7$};

\node[vertex] (v21) at (5,3) {$1$};
\node[vertex] (v22) at (5,1) {$2$};
\node[vertex] (v23) at (5,-1) {$6$};
\node[vertex] (v24) at (5,-3) {$8$};


\draw[edge] (v11) -- (v22);
\draw[edge] (v11) -- (v23);
\draw[edge] (v11) -- (v24);

\draw[edge] (v12) -- (v21);
\draw[edge] (v12) -- (v23);
\draw[edge] (v12) -- (v24);

\draw[edge] (v13) -- (v21);
\draw[edge] (v13) -- (v22);
\draw[edge] (v13) -- (v24);

\draw[edge] (v14) -- (v21);
\draw[edge] (v14) -- (v22);
\draw[edge] (v14) -- (v23);


\node[left=4pt of v11]  {$16$};
\node[left=4pt of v12]  {$15$};
\node[left=4pt of v13]  {$11$};
\node[left=4pt of v14]  {$9$};

\node[right=4pt of v21] {$16$};
\node[right=4pt of v22] {$15$};
\node[right=4pt of v23] {$14$};
\node[right=4pt of v24] {$12$};



\end{tikzpicture}
    \caption{Local distance antimagic labeling of $K_2\times K_4$.}
    \label{fig:K22}
\end{figure}

\item[Subcase 3:]$n=6$.\\
We show the labeling in Figure~\ref{fig:K23}. The numbers inside the
vertices represent the labels, while the numbers outside represent
their corresponding weights. Observe that the labeling induces
$8$ distinct weights. Hence,
$\chi_{ld}(K_2\times K_6)\leq 8=n+2$.

\begin{figure}[!ht]
    \centering
    \begin{tikzpicture}[
    vertex/.style={
        scale=0.5,
        circle,
        draw,
        minimum size=10mm,
        inner sep=0pt,
        font=\Large
    },
    edge/.style={
        thick
    }
]


\node[vertex] (v11) at (0,0)   {$\textbf{3}$};
\node[vertex] (v12) at (2,0)   {$\textbf{4}$};
\node[vertex] (v13) at (4,0)   {$\textbf{7}$};
\node[vertex] (v14) at (6,0)   {$\textbf{12}$};
\node[vertex] (v15) at (8,0)   {$\textbf{6}$};
\node[vertex] (v16) at (10,0)  {$\textbf{8}$};

\node[vertex] (v21) at (0,-3)  {$\textbf{1}$};
\node[vertex] (v22) at (2,-3)  {$\textbf{2}$};
\node[vertex] (v23) at (4,-3)  {$\textbf{5}$};
\node[vertex] (v24) at (6,-3)  {$\textbf{10}$};
\node[vertex] (v25) at (8,-3)  {$\textbf{9}$};
\node[vertex] (v26) at (10,-3) {$\textbf{11}$};


\draw[edge] (v11) -- (v22);
\draw[edge] (v11) -- (v23);
\draw[edge] (v11) -- (v24);
\draw[edge] (v11) -- (v25);
\draw[edge] (v11) -- (v26);

\draw[edge] (v12) -- (v21);
\draw[edge] (v12) -- (v23);
\draw[edge] (v12) -- (v24);
\draw[edge] (v12) -- (v25);
\draw[edge] (v12) -- (v26);

\draw[edge] (v13) -- (v21);
\draw[edge] (v13) -- (v22);
\draw[edge] (v13) -- (v24);
\draw[edge] (v13) -- (v25);
\draw[edge] (v13) -- (v26);

\draw[edge] (v14) -- (v21);
\draw[edge] (v14) -- (v22);
\draw[edge] (v14) -- (v23);
\draw[edge] (v14) -- (v25);
\draw[edge] (v14) -- (v26);

\draw[edge] (v15) -- (v21);
\draw[edge] (v15) -- (v22);
\draw[edge] (v15) -- (v23);
\draw[edge] (v15) -- (v24);
\draw[edge] (v15) -- (v26);

\draw[edge] (v16) -- (v21);
\draw[edge] (v16) -- (v22);
\draw[edge] (v16) -- (v23);
\draw[edge] (v16) -- (v24);
\draw[edge] (v16) -- (v25);


\node[above=4pt of v11] {$37$};
\node[above=4pt of v12] {$36$};
\node[above=4pt of v13] {$33$};
\node[above=4pt of v14] {$28$};
\node[above=4pt of v15] {$29$};
\node[above=4pt of v16] {$27$};

\node[below=4pt of v21] {$37$};
\node[below=4pt of v22] {$36$};
\node[below=4pt of v23] {$33$};
\node[below=4pt of v24] {$28$};
\node[below=4pt of v25] {$34$};
\node[below=4pt of v26] {$32$};



\end{tikzpicture}
    \caption{Local distance antimagic labeling of $K_2\times K_6$.}
    \label{fig:K23}
\end{figure}

\item[Subcase 4:] $n\geq 8$ and $n\equiv 0 \pmod{4}$.\\
Define the following two sequences:
\begin{align*}
A&=(a_i)=\{1+4(i+1);\ 0\leq i \leq \tfrac{n}2-2\},\\
B&=(b_i)=\{2+4i;\ 0\leq i \leq \tfrac{n}2-1\}.
\end{align*}
Using these sequences, define a function
$f\colon V(K_2\times K_n)\to {1,2,\dots,2n}$ by
\begin{align*}
f(x,v_i)&=
\begin{cases}
a_{\tfrac{i-1}2} &\text{if $i\in \left\{1,3,\dots,n-3\right\}$},\\
b_{\tfrac{i}2-1} &\text{if $i\in \left\{2,4,\dots,\tfrac{n}2-2\right\}$},\\
b_{\tfrac{i}2} &\text{if $i\in \left\{\tfrac{n}2,\tfrac{n}2+2,\dots,n-2\right\}$}.
\end{cases}
\end{align*}

Note that the entry $b_{n/4-1}=n-2$ from the sequence $B$ is omitted.
We use this omitted entry to define $f(x,v_{n-1})=n-2,$ and
$f(x,v_n)=n.$

Furthermore, define
$$
f(y,v_i)=f(x,v_i)+2,
\qquad 1\leq i\leq n-2,
$$
and
$$
f(y,v_{n-1})=1,
\qquad
f(y,v_n)=3.
$$

The entries of the sequence $A$ start from $5$ and end at $2n-3$,
while the entries of the sequence $B$ start from $2$ and end at $2n-2$.
Hence, the labels $f(x,v_i)$ and $f(y,v_i)$, for $1\leq i\leq n-2$, take values from
$1$ up to $2n$, with the exception of the labels $1,3,n-2,$ and
$n$. These four labels are precisely assigned to
$$
f(y,v_{n-1}),\quad f(y,v_n),\quad
f(x,v_{n-1}),\quad f(x,v_n),
$$
respectively. Therefore, $f$ is a bijection.

We now determine the sums of the labels in the two parts. We have
\begin{align*}
\sum_{i=0}^{\tfrac{n}{2}-2}a_i
&=
\left(\tfrac{n}{2}-1\right)
+4\left(1+2+\dots+\tfrac{n}{2}-1\right)\
=
\tfrac{n^2-n}{2}-1,
\end{align*}
and
\begin{align*}
\sum_{i=0}^{\tfrac{n}{2}-1}b_i
&=
n
+4\left(1+2+\dots+\tfrac{n}{2}-1\right)\
=
\tfrac{n^2}{2}.
\end{align*}

Consequently,
\begin{align*}
\sum_{i=1}^n f(x,v_i)
&=
\sum_{i=0}^{\tfrac{n}{2}-2}a_i
+\sum_{i=0}^{\tfrac{n}{2}-1}b_i+n\
=
n^2+\tfrac{n}{2}-1
=\alpha,
\end{align*}
and
\begin{align*}
\sum_{i=1}^n f(y,v_i)
&=
\alpha-(n-2+n)+2(n-2)+1+3\
=
\alpha+2=\beta.
\end{align*}

Since $\beta$ is fixed and $f$ is a bijection, for $i\neq j$ we have $w(x,v_i)\neq w(x,v_j).$
Similarly, since $\alpha$ is fixed, $w(y,v_i)\neq w(y,v_j)$
for $i\neq j$.

Furthermore, for $1\leq i\leq n-2$,
$$
f(y,v_i)-f(x,v_i)=2=\beta-\alpha.
$$
Therefore,
$$
w(x,v_i)=w(y,v_i),
\qquad
1\leq i\leq n-2.
$$

For the remaining four vertices, we have
$$
f(x,v_{n-1})=n-2,\qquad
f(x,v_n)=n,
$$
and
$$
f(y,v_{n-1})=1,\qquad
f(y,v_n)=3.
$$

Hence, their corresponding weights are distinct from the common weights
obtained for $1\leq i\leq n-2$. Consequently, the labeling induces
$n-2$ common weights together with four additional distinct weights.
Thus, the total number of distinct weights is $n+2$.

Therefore, $f$ is a local distance antimagic labeling of
$K_2\times K_n$ which induces $n+2$ distinct weights. Hence,
$\chi_{ld}(K_2\times K_n)\leq n+2$.

\item[Subcase 5:] $n\geq 10$ and $n\equiv 2 \pmod{4}$.\

Define the following two sequences:
\begin{align*}
A&=(a_i)=\{2+4(i+1);\ 0\leq i \leq \tfrac{n}{2}-2\},\\
B&=(b_i)=\{1+4i;\ 0\leq i \leq \tfrac{n}{2}-1\}.
\end{align*}
Using these sequences, define a function
$f\colon V(K_2\times K_n)\to {1,2,\dots,2n}$ by
\begin{align*}
f(x,v_i)&=
\begin{cases}
a_{\tfrac{(i-1)}{2}} &\text{if $i\in \left\{1, 3, \dots, n-3\right\}$},\\
b_{\tfrac{i}{2}-1} &\text{if $i\in \left\{2,4,\dots,\tfrac{n}{2}-1\right\}$},\\
b_{\tfrac{i}{2}} &\text{if $i\in \left\{\tfrac{n}{2}+1,\tfrac{n}{2}+3,\dots,n-2\right\}$}.
\end{cases}
\end{align*}

Note that the entry $b_{(n-2)/4}=n-1$ from the sequence $B$ is omitted.
We use this omitted entry to define $f(x,v_{n-1})=n-1,$ and
$f(x,v_n)=n+1.$

Furthermore, define
$$
f(y,v_i)=f(x,v_i)+2,
\qquad 1\leq i\leq n-2,
$$
and
$$
f(y,v_{n-1})=2,
\qquad
f(y,v_n)=4.
$$

The entries of the sequence $A$ start from $6$ and end at $2n-2$,
while the entries of the sequence $B$ start from $1$ and end at $2n-3$.
Hence, the labels $f(x,v_i)$ and $f(y,v_i)$, for $1\leq i\leq n-2$, take values from
$1$ up to $2n$, with the exception of the labels $2,4,n-1,$ and
$n+1$. These four labels are precisely assigned to

$$
f(y,v_{n-1}),\quad f(y,v_n),\quad
f(x,v_{n-1}),\quad f(x,v_n),
$$

respectively. Therefore, $f$ is a bijection.

We now determine the sums of the labels in the two parts. We have
\begin{align*}
\sum_{i=0}^{\tfrac{n}{2}-2}a_i
&=
2\left(\tfrac{n}{2}-1\right)
+4\left(1+2+\dots+\tfrac{n}{2}-1\right)\
=
\tfrac{n^2}{2}-2,
\end{align*}
and
\begin{align*}
\sum_{i=0}^{\tfrac{n}{2}-1}b_i
&=
\tfrac{n}{2}
+4\left(1+2+\dots+\tfrac{n}{2}-1\right)\
=
\tfrac{n^2-n}{2}.
\end{align*}

Consequently,
\begin{align*}
\sum_{i=1}^n f(x,v_i)
&=
\sum_{i=0}^{\tfrac{n}{2}-2}a_i
+\sum_{i=0}^{\tfrac{n}{2}-1}b_i+n+1\
=
\tfrac{n^2}{2}-1
+\tfrac{n^2-n}{2}
+n+1
=\alpha,
\end{align*}
and
\begin{align*}
\sum_{i=1}^n f(y,v_i)
&=
\alpha-(n-1+n+1)+2(n-2)+2+4\
=
\alpha+2=\beta.
\end{align*}

Since $\beta$ is fixed and $f$ is a bijection, for $i\neq j$ we have $w(x,v_i)\neq w(x,v_j).$
Similarly, since $\alpha$ is fixed, $w(y,v_i)\neq w(y,v_j)$
for $i\neq j$.

Furthermore, for $1\leq i\leq n-2$,
$$
f(y,v_i)-f(x,v_i)=2=\beta-\alpha.
$$
Therefore,
$$
w(x,v_i)=w(y,v_i),
\qquad
1\leq i\leq n-2.
$$

For the remaining four vertices, we have
$$
f(x,v_{n-1})=n-1,\qquad
f(x,v_n)=n+1,
$$
and
$$
f(y,v_{n-1})=2,\qquad
f(y,v_n)=4.
$$

Hence their corresponding weights are distinct from the common weights
obtained for $1\leq i\leq n-2$. Consequently, the labeling induces
$n-2$ common weights together with four additional distinct weights.
Thus, the total number of distinct weights is $n+2$.

Therefore, $f$ is a local distance antimagic labeling of
$K_2\times K_n$ which induces $n+2$ distinct weights. Hence, $\chi_{ld}(K_2\times K_n)\leq n+2$.\end{description}\end{description}
\vspace*{-8mm}\end{proof}

Theorem~\ref{th:nodirect} also shows that the neighborhood balanced coloring hypothesis in the subsequent results cannot simply be omitted. Indeed, consider the particular case \(n=3\). The graph \(K_2\) is regular, but it does not admit a neighborhood balanced \(k\)-coloring for any \(k\geq 2\). Moreover,
$
K_2\times K_3\cong C_6$.
Therefore, by Theorem~\ref{th:cycle},
$$
\chi_{ld}(K_2\times K_3)
 =\chi_{ld}(C_6)
 =4
 >2
 =\chi_{ld}(K_2).
$$
Thus, the inequality
$$
\chi_{ld}(G\times H)\leq\chi_{ld}(G)
$$
does not hold for arbitrary regular graphs. Consequently, an additional structural assumption is required; in Lemmas~\ref{reg-even-nbc} and~\ref{reg-odd-nbc} and Theorem~\ref{th:even-odd-nbc}, this role is played by the neighborhood balanced coloring hypothesis.


The next two lemmas provide an upper bound (analogous to Hedetniemi's conjecture) for the local distance antimagic chromatic number for the direct product of two regular graphs, one of which is a \nbcd graph.  

\begin{lemma}\label{reg-even-nbc}
    Let $G$ be a regular \nbcd graph with  $k$ even,  and let $H$ be a regular graph. Then $\chi_{ld}(G\times H)\leq \chi_{ld}(G)$.
\end{lemma}
\begin{proof}
 Since the graph $G$ admits a \nbc with  $k$ even, it follows from Lemma \ref{reduction}, that $G$ admits a neighborhood-balanced $2$-coloring. Furthermore, by Lemma \ref{th:prep1}, as $G$ is a regular graph admitting a neighborhood-balanced $2$-coloring, the degree of every vertex must be a multiple of $2$, let this degree be $2t$. Furthermore, let   $H$ be an $r$-regular graph.
 
 Denote the vertex sets of $G$ and $H$ by $V(G)=\{x_1,x_2,\dots,x_n\}$ and $V(H)=\{y_1,y_2,\dots,y_{m}\}$, respectively. Then the vertex sets of the direct product $G\times H$ is $V(G\times H)=\{(x_i,y_j): 1 \leq i \leq n,  1\leq j \leq m\}$. Let $f$ be a local distance antimagic labeling of $G$ that induces $\chi_{ld}(G)$ distinct weights on the vertices of $G$, and let  $h$ be a neighborhood balanced $2$-coloring of $G$. Using these mappings $f$ and $h$ we define a bijection  $g\colon V(G\times H)\to \{1,2,\dots,mn\}$ by
		for $1\leq i \leq n$ and $1\leq j \leq m$,
		\begin{align*}
			g(x_i,y_j)=\begin{cases}
				m(f(x_{i})-1)+j &\text{if $h(x_{i})=1$},\\
				mf(x_{i})+1-j &\text{if $h(x_{i})=2$}.
			\end{cases}
		\end{align*}
		For $1\leq i \leq n$ and $1\leq j \leq m$, the weight of vertices are given by
        \begin{align*}
         w_{G\times H}(x_i,y_j)&= \sum_{\substack{x_s\in N_G(x_i),\\ y_q\in N_H(y_j)}}g(x_s,y_q)=\sum_{y_q\in N_H(y_j)}\left(\sum_{x_s\in N_G(x_i)}g(x_s,y_q)\right)\\
          &=\sum_{y_q\in N_H(y_j)}\left(\sum_{\substack{x_s\in N_G(x_i),\\ h(x_s)=1}}g(x_s,y_q)+\sum_{\substack{x_s\in N_G(x_i),\\h(x_s)=2}}g(x_s,y_q)\right)\\
          &=\sum_{y_q\in N_H(y_j)}\left(\sum_{\substack{x_s\in N_G(x_i),\\ h(x_s)=1}}\left(m(f(x_s)-1)+q\right)+\sum_{\substack{x_s\in N_G(x_i),\\ h(x_s)=2}}\left(mf(x_s)+1-q\right)\right)\\
          &=\sum_{y_q\in N_H(y_j)}\left(mw_f(x_i)+(1-m)t\right)=r\left(mw_f(x_i)+(1-m)t\right).
        \end{align*}
         Since $f$ is a local distance antimagic labeling of $G$, any two  adjacent vertices in $G\times H$ have distinct weights.

         Consequently, adjacent vertices in $G\times H$  also receive distinct weights under the labeling $g$.
Therefore, $g$ is a local distance antimagic labeling of $G\times H$, inducing exactly $\chi_{ld}(G)$ distinct weights to the vertices of $G$. Hence, $\chi_{ld}(G\times H)\leq \chi_{ld}(G)$.
\end{proof}
\begin{lemma}\label{reg-odd-nbc}
    Let $G$ be a regular \nbcd graph with  $k$ odd and  let $H$ be a regular graph of odd order. Then $\chi_{ld}(G\times H)\leq \chi_{ld}(G)$.
\end{lemma}
\begin{proof}
 Since the graph $G$ admits a \nbc with  $k$ odd, it follows from Lemma \ref{reduction} that $G$ also admits a neighborhood balanced $p$-coloring, where $p$ is the smallest odd prime divisor of $k$. Furthermore, by Lemma \ref{th:prep1}, since $G$ is a regular graph admitting a neighborhood balanced $p$-coloring, the degree of each vertex must be a multiple of $p$, let this degree be $pt$. Let   $H$ be an $r$-regular graph of odd order $m$.
 
 Denote the vertex sets of $G$ and $H$ by $V(G)=\{x_1,x_2,\dots,x_n\}$ and $V(H)=\{y_1,y_2,\dots,y_{m}\}$, respectively. Then the vertex set of the direct product $G\times H$ is $V(G\times H)=\{(x_i,y_j): 1 \leq i \leq n, 1\leq j \leq m\}$. Let $f$ be a local distance antimagic labeling of $G$ that assigns $\chi_{ld}(G)$ distinct weights to the vertices of $G$, and let $h$ be a neighborhood balanced $p$-coloring of $G$. We now define a bijection $g\colon V(G\times H)\to \{1,2,\dots,mn\}$ as follows. For $1\leq i \leq n$
		\begin{equation*}
			g(x_i,y_j)=\begin{cases}
				m(f(x_{i})-1)+j &\text{if $h(x_{i})=1$ and $1\leq j \leq m$},\\
				m(f(x_i)-1)+\tfrac{m+1}{2}+j &\text{if $h(x_{i})=2$ and $1\leq j \leq \tfrac{m-1}{2}$},\\
                m(f(x_i)-1)-\tfrac{m-1}{2}+j &\text{if $h(x_{i})=2$ and $\tfrac{m+1}{2}\leq j \leq m$},\\
                
                mf(x_i)+1-2j&\text{if $h(x_{i})=3$ and $1\leq j \leq \tfrac{m-1}{2}$},\\
                mf(x_i)+m+1-2j &\text{if $h(x_{i})=3$ and $\tfrac{m+1}{2}\leq j \leq m$},\\
                m(f(x_i)-1)+j &\text{if $h(x_i)=l$, $4\leq l \leq \tfrac{p+3}{2}$ and $1\leq j \leq m$},\\
                 mf(x_i)+1-j &\text{if $h(x_i)=l$, $\tfrac{p+5}{2}\leq l \leq p$ and $1\leq j \leq m$.}
			\end{cases}
		\end{equation*}
For  $1\leq i \leq n$ and $1\leq j \leq {(m-1)}/{2}$, the weight of vertex $(x_i,y_j)$ is
\begin{align*}
 w_{G\times H}(x_i,y_j)=& \sum_{\substack{x_s\in N_G(x_i), \\ y_q\in N_H(y_j)}}g(x_s,y_q)=\sum_{y_q\in N_H(y_j)}\sum_{x_s\in N_G(x_i)}g(x_s,y_q)\\
          =&\sum_{y_q\in N_H(y_j)}\left(\sum_{\substack{x_s\in N_G(x_i),\\ h(x_s)=1}}g(x_s,y_q)+\sum_{\substack{x_s\in N_G(x_i),\\ h(x_s)=2}}g(x_s,y_q)+\dots +\sum_{\substack{x_s\in N_G(x_i),\\  h(x_s)=p}}g(x_s,y_q)\right)\\ 
          =&\sum_{y_q\in N_H(y_j)}\left(\sum_{\substack{x_s\in N_G(x_i),\\ h(x_s)=1}}g(x_s,y_q)+\sum_{\substack{x_s\in N_G(x_i),\\ h(x_s)=2}}g(x_s,y_q)+\sum_{\substack{x_s\in N_G(x_i),\\ h(x_s)=3}}g(x_s,y_q)\right.\\ & \left. +\sum_{l=4}^{\tfrac{p+3}{2}} \sum_{\substack{x_s\in N_G(x_i),\\ h(x_s)=l}}g(x_s,y_q) + \sum_{l=\tfrac{p+5}{2}}^{p} \sum_{\substack{x_s\in N_G(x_i),\\ h(x_s)=l}}g(x_s,y_q)\right)\\
          =&\sum_{y_q\in N_H(y_j)}\left(\sum_{\substack{x_s\in N_G(x_i),\\ h(x_s)=1}}\left(m(f(x_s)-1)+q\right)+\sum_{\substack{x_s\in N_G(x_i),\\ h(x_s)=2}}\left(m(f(x_s)-1)+q+\tfrac{m+1}{2}\right)\right.\\&+\sum_{\substack{x_s\in N_G(x_i),\\ h(x_s)=3}}\left(mf(x_s)+1-2q\right)+\sum_{l=4}^{\tfrac{p+3}{2}}\sum_{\substack{x_s\in N_G(x_i),\\ h(x_s)=l}}\left(m(f(x_s)-1)+q\right)\\
          &\left. +\sum_{l=\tfrac{p+5}{2}}^{p}\sum_{\substack{x_s\in N_G(x_i),\\ h(x_s)=l}}\left(mf(x_s)+1-q\right)
          \right)=\sum_{y_q\in N_H(y_j)}\left(m
          \sum_{x_s\in N_G(x_i)}f(x_s)+\tfrac{(1-m)pt}{2}
          \right)
          \\
         =&\sum_{y_q\in N_H(y_j)}\left(m\ w_f(x_i)+\tfrac{(1-m)pt}{2}
          \right)
         =r\left(m\ w_f(x_i)+\tfrac{(1-m)pt}{2}\right).
          \end{align*}
        In a similar manner, we can show that for $1\leq i \leq n$ and $\tfrac{m+1}{2}\leq j \leq m$,
        $$w_{G\times H}(x_i,y_j)=r\left(m\ w_f(x_i)+\tfrac{(1-m)pt}{2}\right).$$
         Since $f$ is a local distance antimagic labeling,   adjacent vertices of $G\times H$ receive distinct weights, and hence $g$ is a local distance antimagic labeling of $G\times H$ that assigns $\chi_{ld}(G)$ distinct weights to its vertices. 
\end{proof}
The two lemmas just proved can be combined to yield the following theorem.
\begin{theorem}\label{th:even-odd-nbc}
Let $G$ be a regular neighborhood balanced $k$-colored graph and let $H$ be a regular graph of odd order. Then $\chi_{ld}(G\times H)\leq \chi_{ld}(G)$.    
\end{theorem}
Lemma \ref{reg-odd-nbc} establishes a local distance antimagic labeling of the direct product of a regular \nbcd graph with a regular graph of odd order. However, we do not yet have a corresponding result for the direct product of a regular \nbcd graph with a regular graph of even order. In the following, we present an example of a regular \nbcd graph for which the local distance antimagic chromatic number of its direct product with a regular graph of even order is bounded above by its own local distance antimagic chromatic number. Let $K_{t[n]}$ denote the complete $t$-partite graph whose vertex set is partitioned into $t\ge 2$ disjoint independent sets (parts), each  of cardinality $n$, such that two vertices are adjacent if and only if they belong to distinct parts.
\begin{theorem}
   For positive integers $n\ge 3$ odd and $t$ even, $\chi_{ld}( K_{t[n]}\times H)\leq\chi_{ld}(K_{t[n]})= t$, where $H$ is a regular graph of even order.
\end{theorem}
\begin{proof}
  
 Let $G=K_{t[n]}$  and let $H$ be an $r$-regular graph of even order $m$, where $r\geq 2$.
  Denote the vertex set of $G$ and $H$ by $V(G)=\{x_{i,j}: 1\leq i \leq n,  1\leq j\leq t\}$ and $V(H)=\{y_1,y_2,\dots,y_{m}\}$, respectively. Then the vertex set of the direct product $G\times H$ is $V(G\times H)=\{(x_{i,j},y_h): 1 \leq i \leq n, 1\leq j \leq t,  1\leq h \leq m\}$.
 
  Define a bijection $f\colon V(G \times H)\to \{1,2,\dots,mtn\}$ as follows. For $1\leq h \leq m$, 
    \begin{align*}
        f(x_{i,j}, y_h)&=
        \begin{cases}
            (i-1)m+h+3(j-1)m &\text{if $i=1,2$ and $j=1,3,\dots,t-1$},\\ 
            4m-(2h-1)+3(j-1)m &\text{if $i=3$ and $j=1,3,\dots,t-1$},\\ 
             4m+(j-2)3m-2(h-1) &\text{if $i=1$ and $j=2,4,\dots,t$},\\ 
            4m+(j-2)3m+h+(i-2)m &\text{if $i=2,3$ and $j=2,4,\dots,t$},\\ 
            (i-1)tm+(h-1)t+j &\text{if $i=4,6,\dots,n-1$ and $j=1, 2, \dots, t$},\\ 
            itm-(h-1)t-j+1 &\text{if $i=5,7,\dots,n$ and $j=1, 2, \dots, t$}.
        \end{cases}
    \end{align*}
    Next, fix $1\leq h \leq m$. We now calculate the sum $\sum_{i=1}^nf(x_{i,j},y_h)$ for a fixed $j$. The computation naturally divides into two cases, according to the parity of $j$.
    \begin{description}
    \item[Case 1:] when $j$ is odd.
    \begin{align*}
\sum_{i=1}^3 f(x_{i,j},y_h)
&= (h + 3(j-1)m) + (m + h + 3(j-1)m) + (4m - (2h-1) + 3(j-1)m) \\
&= 9jm - 4m + 1,\\
\sum_{i=4}^n f(x_{i,j},y_h)
&=  \tfrac{n-3}{2} \left((n+3)tm + 1\right).
\end{align*}
Hence,
\[
\sum_{i=1}^n f(x_{i,j},y_h) = 9jm - 4m + 1 +  \tfrac{n-3}{2} \left((n+3)tm + 1\right).
\]
\item[Case 2:]  when $j$ is even.

Similarly we obtain
\[
\sum_{i=1}^3 f(x_{i,j},y_h) = 9jm - 5m + 2,
\]
and thus
\[
\sum_{i=1}^n f(x_{i,j},y_h) = 9jm - 5m + 2 +  \tfrac{n-3}{2} \left((n+3)tm + 1\right).
\]
\end{description}
Observe that these sums are independent of the index $h$.  For $1\leq i \leq n$ and $1\leq j \leq t$, and $1\leq h \leq m$, the weights of the vertices are given as follows.
\begin{description}
    \item[Case 1:]  when $j$ is odd.
    \begin{align*}
\hspace*{-15mm} w_{G\times H}(x_{i,j},y_h)=& \sum_{\substack{x_{p,s}\in N_G(x_{i,j}),\\  y_q\in N_H(y_h)}}f(x_{p,s},y_q)
        =\sum_{y_q\in N_H(y_h)} \sum_{\substack{s=1,\\ s\not=j}}^t\ \sum_{p=1}^nf(x_{p,s},y_q) \\
         = \; &\sum_{y_q\in N_H(y_h)}\Big(\tfrac{n-3}{2}\left((n+3)tm+1\right)(t-1)+ \tfrac{t-2}{2}(-9m+3)   -5m+2
         \\&\left. +9m\left(\tfrac{t(t+1)}{2}-j\right)\!\right)   \\
         = \; &r\Big(\tfrac{n-3}{2}\left((n+3)tm+1\right)(t-1)+ \tfrac{t-2}{2}(-9m+3)   -5m+2+9m\left(\tfrac{t(t+1)}{2}-j\right)\!\!\Big).
\end{align*}
\item[Case 2:] when $j$ is even. 
    \begin{align*}
\hspace*{-15mm}  w_{G\times H}(x_{i,j},y_h)=& \sum_{\substack{x_{p,s}\in N_G(x_{i,j}),\\  y_q\in N_H(y_h)}}f(x_{p,s},y_q) 
 =\sum_{y_q\in N_H(y_h)} \sum_{\substack{s=1,\\ s\not=j}}^t\ \sum_{p=1}^nf(x_{p,s},y_q)\\
         = \; &\sum_{y_q\in N_H(y_h)}\Big(\tfrac{n-3}{2}\left((n+3)tm+1\right)(t-1)+\tfrac{t-2}{2}(-9m+3)-4m+1\\ & \left. +9m\left(\tfrac{t(t+1)}{2}-j\right)\!\right)\\
         = \; &r\Big(\tfrac{n-3}{2}\left((n+3)tm+1\right)(t-1)+\tfrac{t-2}{2}(-9m+3)-4m+1  +9m\left(\tfrac{t(t+1)}{2}-j\right)\!\!\Big).
\end{align*}
\end{description}
From the computed weights, it is evident that adjacent vertices in $G\times H$ have distinct weights and thus $f$ is a local distance antimagic labeling of $G\times H$ that assigns $t$ distinct weights to its vertices.
\end{proof}

The following theorem provides conditions under which the local distance antimagic chromatic number of the direct product of two regular neighborhood balanced colored graphs is bounded above by the smaller of the local distance antimagic chromatic numbers of the factor graphs. The result follows directly by applying Lemmas~\ref{reg-even-nbc} and \ref{reg-odd-nbc} to both factors. It provides a local distance antimagic analogue of the upper bound suggested by Hedetniemi's conjecture for the ordinary chromatic number.

\begin{theorem}\label{direct-product-nbc}
    Let $G$ be a regular neighborhood balanced $k$-colored graph and  let  $H$ be a regular neighborhood balanced $l$-colored graph. Then $\chi_{ld}(G\times H)\leq \min\{\chi_{ld}(G), \chi_{ld}(H)\}$ in each of the following cases:
    \begin{enumerate}
        \item both $k$ and $l$ are even.
        \item $k$ and  $|V(H)|$ are odd, and $l$ is even.
        \item $l$ and $|V(G)|$ are odd, and $k$ is even.
        \item $k$, $l$, $|V(G)|$, and $|V(H)|$ are all odd.
    \end{enumerate}
\end{theorem}

Theorem \ref{direct-product-nbc} does not address the following configurations of $G$ and $H$:
\begin{enumerate}
    \item $k$ and $l$ are both odd, and at least one of $|V(G)|$ or $|V(H)|$ is even.
    \item $l$ and $|V(H)|$ are both even, while $k$ is odd.
\end{enumerate}

In the sequel, we present families of regular neighborhood balanced colored graphs $G$ and $H$  that cover the omitted configurations and satisfy
$\chi_{ld}(G\times H)\leq
\min\{\chi_{ld}(G),\chi_{ld}(H)\}$.

It follows from Theorem~\ref{directwithbipartite} that the local distance antimagic chromatic number of the direct product of a regular graph with \(K_{m,m}\) is equal to \(2\) for every positive integer \(m\geq 2\). 
Furthermore, if $n\geq 3$ is odd, then, by Theorem~\ref{nbc-multipartite}, the regular complete multipartite graph $K_{t[n]}$ admits a neighborhood balanced $n$-coloring. Consequently, for positive integers $m\geq 2$, $t\geq 2$, and $n\geq 3$ odd, we have
\[
\chi_{ld}\bigl(K_{m,m}\times K_{t[n]}\bigr)=2.
\]



\subsection{Local Distance Antimagic Labeling of Union of Graphs}\label{sec:union}
We have observed that the local distance antimagic chromatic number of the direct product of regular \nbcd graphs satisfies a relation analogous to Hedetniemi’s conjecture for the chromatic number.

For a graph $G$ and a positive integer $m$, let $mG$ denote the disjoint union of $m$ copies of   $G$.
	It is known that $\chi(mG)=\chi(G)$, where $\chi$ denotes the chromatic number of $G$. A natural direction is to check if, for a graph $G$, $\chi_{ld}(mG)=\chi_{ld}(G)$.
This, however, is not true in general. We present an example to affirm the same.
\begin{theorem}\label{th:nounion}
    For positive integers $m\geq 2$ and $n\geq 2$,
    \[
    \chi_{ld}(mK_n)\geq n+1>\chi_{ld}(K_n).
    \]
\end{theorem}

\begin{proof}
Let $V(K_n)=\{v_1,v_2,\dots,v_n\}$ and
$V(mK_n)=\{v_i^j:1\leq i\leq n,\;1\leq j\leq m\},$
where $v_1^j,v_2^j,\dots,v_n^j$ are the vertices of the $j$th copy of $K_n$.

Clearly,
\[
\chi_{ld}(mK_n)\geq \chi(mK_n)=\chi(K_n)=n.
\]
Suppose, to the contrary, that $\chi_{ld}(mK_n)=n.$
Let $f$ be a local distance antimagic labeling of $mK_n$ that induces exactly $n$ distinct vertex weights. For $1\leq j\leq m$, let $$S_j=\sum_{i=1}^n f(v_i^j)$$
be the sum of the labels assigned to the vertices in the $j$th copy of $K_n$.

Since the vertices within each copy of $K_n$ are pairwise adjacent, their weights must be pairwise distinct. Thus, each copy of $K_n$ contains all the $n$ distinct weights induced by $f$. Hence, after renumbering the vertices in each copy if necessary, we may assume that $w(v_i^j)=w(v_i^k)$
for every $1\leq i\leq n$ and $j\neq k$.

Since the neighbors of $v_i^j$ are precisely the other vertices in the same copy of $K_n$, we have
\[
w(v_i^j)=S_j-f(v_i^j).
\]
Therefore, for any $j\neq k$,
\begin{align}\label{eq:nounion1}
    S_j-f(v_i^j)
    &=S_k-f(v_i^k),
\end{align}
and hence
\[
S_j-S_k=f(v_i^j)-f(v_i^k).
\]
Summing over $i=1,2,\dots,n$, we obtain
\begin{align*}
    n(S_j-S_k)
    &=\sum_{i=1}^n\bigl(f(v_i^j)-f(v_i^k)\bigr)
    =S_j-S_k.
\end{align*}
Thus,
\[
(n-1)(S_j-S_k)=0.
\]
Since $n\geq2$, it follows that $S_j=S_k.$
Equation \eqref{eq:nounion1} now gives $f(v_i^j)=f(v_i^k)$
for every $i$, which contradicts the fact that $f$ is a bijection.

Therefore,
\[
\chi_{ld}(mK_n)\geq n+1>\chi_{ld}(K_n)=\chi(K_n)=n.
\]
\end{proof}

\begin{cor}\label{cr:nounion}
    For a positive integer $m$, $\chi_{ld}(mK_2)=2m$.
\end{cor}
\begin{proof}
 Let $V(mK_2)=\{v_1^j,v_2^j;\ 1\leq j \leq m\}$ be the vertex set of $mK_2$. Let $f\colon V(mK_2)\to \{1,2,\dots,2m\}$ be any local distance antimagic labeling of $mK_2$.  

 Note that for any $1\leq j \leq m$, we have 
 \begin{align*}
     w(v_1^j)&=f(v_2^j),\\
      w(v_2^j)&=f(v_1^j).
 \end{align*}
  As $f$ is a bijection, all the vertices must receive distinct weights. Thus $$\chi_{ld}(mK_2)\geq 2m.$$
  Further, as $|V(mK_2)|=2m$, the upper bound is obvious and thus the equality follows.
\end{proof}

Theorem~\ref{th:nounion} shows that the neighborhood balanced coloring hypothesis in the subsequent results concerning disjoint unions cannot simply be omitted. Indeed, although the graph $K_n$ is $(n-1)$-regular, for every $m,n\geq 2$ we have
\[
\chi_{ld}(mK_n)\geq n+1>\chi_{ld}(K_n)=\chi(K_n)=n.
\]
Thus, the inequality
\[
\chi_{ld}(mG)\leq\chi_{ld}(G)
\]
does not hold for arbitrary regular graphs. Consequently, an additional structural assumption is required, in Lemmas~\ref{union-even-nbc} and~\ref{union-odd-nbc-odd-copies} and Theorem~\ref{union-nbc}, this role is played by the neighborhood balanced coloring hypothesis.

Theorem \ref{th:nounion} naturally leads to the question of characterizing graphs $G$ for which $\chi_{ld}(mG)=\chi_{ld}(G)$ for every positive integer $m$.
Corollary~\ref{cr:nounion} shows that the local distance antimagic chromatic number of a disjoint union may be strictly larger than that of the original graph, even when the original graph contains pendant vertices. In particular, for $G=K_2$, this number increases with each additional copy of $G$.

This observation motivates a separate study of the local distance antimagic chromatic number of graphs containing pendant vertices.
In this section, we present results in both the above-mentioned directions.

    Since $\chi(G)$ provides a natural  lower bound for $\chi_{ld}(mG)$, the next result establishes that regular graphs admitting neighborhood balanced $k$-coloring satisfy the inequality  $\chi_{ld}(mG)\leq \chi_{ld}(G)$.
\begin{lemma}\label{union-even-nbc}
 Let $m$ be a positive integer and let $G$ be a regular neighborhood balanced $k$-colored graph with $k$ even. Then $\chi_{ld}(mG)\leq \chi_{ld}(G)$.   
\end{lemma}
\begin{proof}
   Since $G$ is a neighborhood balanced $k$-colored graph with $k$-even, by Lemma \ref{reduction}, $G$ is also a neighborhood balanced $2$-colored graph. Assume that $G$ is a $2r$-regular graph  admitting neighborhood balanced $2$-coloring $h$ and a local distance antimagic labeling $f$ that assigns $\chi_{ld}(G)$ distinct weights to its vertices. As $h$ is a neighborhood balanced $2$-coloring of $G$, we have, for every vertex $u\in V(G)$, $$\sum_{\substack{v\in N_G(u),\\ h(v)=1}}1=\sum_{\substack{v\in N_G(u),\\  h(v)=2}}1=r.$$ 
   Now consider the graph $mG$.
   Let $V(mG)=\{v^i: 1\leq i \leq m, v\in V(G)\}$ be its vertex set.
		Define a vertex labeling  $g\colon V(mG)\to\{1,2,\dots,m|V(G)|\}$ in the following way. For $1\leq i \leq m$
		\begin{equation*}
			g(v^{i})=
			\begin{cases}
				m(f(v)-1)+i &\text{if $h(v)=1$},\\
				mf(v)+1-i &\text{if $h(v)=2$}.
			\end{cases}
		\end{equation*}
	Note that if a vertex in $G$ receives the label $t\in \{1,2,\dots,|V(G)|\}$ under the labeling $f$, then the corresponding vertices in $mG$ receive labels from the set $\{m(t-1)+1, m(t-1)+2,\dots, mt\}$ under $g$.
		Moreover, for the weight of a vertex $u^{i}$, $1\leq i \leq m$, in $mG$, under the labeling $g$, we have
		\begin{align*}
			w_{g}(u^{i})&=\sum_{v\in N_G(u)}g(v^i)=\sum_{\substack{v\in N_G(u),\\ h(v)=1}}g(v^i)+\sum_{\substack{v\in N_G(u),\\ h(v)=2}}g(v^i) \\&=\sum_{\substack{v\in N_G(u),\\ h(v)=1}}(m(f(v)-1)+i)+\sum_{\substack{v\in N_G(u),\\ h(v)=2}}(mf(v)+1-i)\\
			&=m\sum_{\substack{v\in N_G(u),\\ h(v)=1}} f(v)+(i-m) \sum_{\substack{v\in N_G(u),\\  h(v)=1}}1 +m\sum_{\substack{v\in N_G(u),\\ h(v)=2}} f(v) +(1-i) \sum_{\substack{v\in N_G(u), \\ h(v)=2}}1 \\
			&=m\ w_{f}(u)+(i-m)r+(1-i)r =m\ w_{f}(u)+(1-m)r.
		\end{align*}
 Since every vertex has $r$ neighbors in each color class, and $f$ is a local distance antimagic labeling of $G$,   adjacent vertices in $G$ have distinct weights. Consequently, adjacent vertices in $mG$ also receive distinct weights. Moreover, vertices corresponding to the same vertex $v$ across the $m$ copies of $G$  have equal weights. Therefore, $g$ is a local distance antimagic labeling of $mG$ that assigns $\chi_{ld}(G)$ distinct weights to the vertices of $mG$. Hence, $\chi_{ld}(mG)\leq \chi_{ld}(G)$.
\end{proof}
\begin{lemma}\label{union-odd-nbc-odd-copies}
 Let $m$ be an odd positive integer and let $G$ be a regular neighborhood balanced $k$-colored graph with $k$ odd. Then $\chi_{ld}(mG)\leq \chi_{ld}(G)$.   
\end{lemma}
\begin{proof}
   As $G$ is a neighborhood balanced $k$-colored graph with $k$ odd. By Lemma \ref{reduction}, $G$ is also a neighborhood balanced $p$-colored graph, where $p$ is the least prime divisor of $k$. Suppose that  $G$ is a $pr$-regular graph  admitting neighborhood balanced $p$-coloring $h$ and a local distance antimagic labeling $f$ that assigns $\chi_{ld}(G)$ distinct weights to its vertices. Furthermore, since $h$ is a neighborhood balanced $p$-coloring of $G$, we have, for every vertex $u\in V(G)$, $$\sum_{\substack{v\in N_G(u), \\ h(v)=1}}1=\sum_{\substack{v\in N_G(u),\\  h(v)=2}}1=\dots=\sum_{\substack{v\in N_G(u), \\ h(v)=p}}1=r.$$ 
   
   Now consider the graph $mG$.
   Let $V(mG)=\{v^i: 1\leq i \leq m, v\in V(G)\}$ be its vertex set. Define a vertex labeling  $g\colon V(mG)\to\{1,2,\dots,m|V(G)|\}$   as follows: 
		\begin{equation*}
			g(v^i)=\begin{cases}
				m(f(v)-1)+i &\text{if $h(v)=1$ and $i=1, 2, \dots, m$},\\
				m(f(v)-1)+\tfrac{m+1}{2}+i &\text{if $h(v)=2$ and $i=1, 2, \dots, \tfrac{m-1}{2}$},\\
                m(f(v)-1)-\tfrac{m-1}{2}+i &\text{if $h(v)=2$ and $i=\tfrac{m+1}{2}, \tfrac{m+3}{2},\dots, m$},\\
                
                mf(v)+1-2i&\text{if $h(v)=3$ and $i=1, 2, \dots,\tfrac{m-1}{2}$},\\
                mf(v)+m+1-2i &\text{if $h(v)=3$ and $i=\tfrac{m+1}{2}, \tfrac{m+3}{2}, m$},\\
                m(f(v)-1)+i &\text{if $h(v)=l$, $l=4,5, \dots, \tfrac{p+3}{2}$ and $i=1, 2, , \dots, m$},\\
                 mf(v)+1-i &\text{if $h(v)=l$, $l=\tfrac{p+5}{2}, \tfrac{p+7}{2},\dots, p$ and $i=1, 2, , \dots, m$.}
			\end{cases}
		\end{equation*}
	Note that  if a vertex in $G$ is labeled with the number $t$, $t\in\{1,2,\dots,|V(G)|\}$ by $f$, then the corresponding vertices in $mG$ receive the labels from the set $\{m(t-1)+1, m(t-1)+2,\hdots, mt\}$ under $g$.
    
		We  now calculate the weight of a vertex $u^i$ under the labeling $g$. For $1\leq i \leq {(m-1)}/{2}$, we have
		\begin{align*}
			w_{g}(u^{i})=&\sum_{v\in N_G(u)}g(v^i)=\sum_{\substack{v\in N_G(u),\\ h(v)=1}}g(v^i)+\sum_{\substack{v\in N_G(u),\\ h(v)=2}}g(v^i)+\dots+\sum_{\substack{v\in N_G(u),\\  h(v)=p}}g(v^i)\\
		=&\sum_{\substack{v\in N_G(u),\\ h(v)=1}}\left(m(f(v)-1)+i\right)+\sum_{\substack{v\in N_G(u),\\ h(v)=2}}\left(m(f(v)-1)+i+\tfrac{m+1}{2}\right)+\sum_{\substack{u\in N_G(v),\\ h(v)=3}}\left(mf(v)+1-2i\right)\\&+\sum_{l=4}^{\tfrac{p+3}{2}} \sum_{\substack{v\in N_G(u),\\ h(v)=l}}\left(m(f(v)-1)+i\right) \sum_{l=\tfrac{p+5}{2}}^{p} \sum_{\substack{v\in N_G(u),\\ h(v)=l}}\left(mf(v)+1-i\right) 
         \\=& \ m         \sum_{v\in N_G(u)}f(v) +\tfrac{(1-m)p}{2} r
          =m\ w_f(u)+\tfrac{(1-m)pr}{2}.
		\end{align*}
Similarly, for $ ({m+1})/{2}\leq i \leq m$,
    \begin{align*}
			w_{g}(u^{i})=&\sum_{v\in N_G(u)}g(v^i)=\sum_{\substack{v\in N_G(u),\\ h(v)=1}}g(v^i)+\sum_{\substack{v\in N_G(u),\\ h(v)=2}}g(v^i)+\dots+\sum_{\substack{v\in N_G(u),\\  h(v)=p}}g(v^i)\\
		=&\sum_{\substack{v\in N_G(u),\\ h(v)=1}}\left(m(f(v)-1)+i\right)+\sum_{\substack{v\in N_G(u),\\ h(v)=2}}\left(m(f(v)-1)-\tfrac{m-1}{2}+i\right)\\&+\sum_{\substack{u\in N_G(v),\\ h(v)=3}}\left(mf(v)+m+1-2i\right)+\sum_{l=4}^{\tfrac{p+3}{2}} \sum_{\substack{v\in N_G(u),\\ h(v)=l}}\left(m(f(v)-1)+i\right)\\&+\sum_{l=\tfrac{p+5}{2}}^{p} \sum_{\substack{v\in N_G(u),\\ h(v)=l}}\left(mf(v)+1-i\right)
           =m 
          \sum_{v\in N_G(u)}f(v) +\tfrac{(1-m)p}{2}r
         =m\ w_f(u)+\tfrac{(1-m)pr}{2}.
		\end{align*}
    The above calculations follow from the fact that every vertex of $G$ has exactly $r$ neighbors in each color class.
     Since $f$ is a local distance antimagic labeling of $G$,   adjacent vertices of $G$ have distinct weights. Consequently,  adjacent vertices of $mG$ also have distinct weights under $g$. Moreover, vertices corresponding to the same vertex of $G$, i.e., belonging to different copies,  have identical weight. Hence, $g$ is a local distance antimagic labeling for $mG$ that assigns  $\chi_{ld}(G)$ distinct weights to the vertices of $mG$. Therefore, $\chi_{ld}(mG)\leq \chi_{ld}(G)$.
\end{proof}
We now combine Lemma \ref{union-even-nbc} and Lemma \ref{union-odd-nbc-odd-copies} to obtain the following theorem.
\begin{theorem}\label{union-nbc}
Let $m$ be a positive integer, and let $G$ be a regular \nbcd graph. Then  $\chi_{ld}(mG)\leq \chi_{ld}(G)$ in each of the following cases: \begin{enumerate}
        \item $k$ is even,
        \item both $m$ and $k$ are odd. 
    \end{enumerate}  
\end{theorem}
We now provide examples of graphs for which the equality in the preceding theorem is attained.
\begin{cor}
Let $m$ be a positive integer, and let $G$ be  a regular \nbcd graph $G$ satisfying $\chi_{ld}(G)=\chi(G)$. Then  $\chi_{ld}(mG)= \chi_{ld}(G)$ in each of the following cases: \begin{enumerate}
        \item $k$ is even,
        \item both $m$ and $k$ are odd. 
    \end{enumerate}  
\end{cor}
    

\begin{cor}
Let $m\ge 1$, $n\ge 2$, and $t\ge 2$ be  integers. Then  $\chi_{ld}(mK_{{t[n]}})=t$ in each of the following cases: \begin{enumerate}
        \item $n$ is even,
        \item both $m$ and $n$ are odd. 
    \end{enumerate}  
\end{cor}

Observe that Lemma \ref{union-odd-nbc-odd-copies} investigates the local distance antimagic chromatic number of odd copies of \nbcd graphs, where $k$  is odd. In the following lemma, we give an upper bound for the local distance antimagic chromatic number of even copies of such a graph. However, the bound obtained in this case is not as sharp as in the preceding result.
\begin{lemma}\label{union-odd-nbc-even-copies}
 Let $m$ be an even positive integer and let $G$ be a regular neighborhood balanced $k$-colored graph with $k$ odd. Then $\chi_{ld}(mG)\leq 2\chi_{ld}(G)$.   
\end{lemma}
\begin{proof}
   Let $G$ be a neighborhood balanced $k$-colored graph with $k$ odd. By Lemma \ref{reduction}, $G$ is also a neighborhood balanced $p$-colored graph, where $p$ is the least prime integer dividing $k$. Suppose that  $G$ is a $pr$-regular graph, admitting neighborhood balanced $p$-coloring $h$ and a local distance antimagic labeling $f$ that assigns $\chi_{ld}(G)$ distinct weights to its vertices. Since $h$ is a neighborhood balanced $p$-coloring of $G$, we have, for every vertex $u\in V(G)$, $$\sum_{\substack{v\in N_G(u),\\ h(v)=1}}1=\sum_{\substack{v\in N_G(u),\\  h(v)=2}}1=\dots=\sum_{\substack{v\in N_G(u), \\ h(v)=p}}1=r.$$ 
   
   Consider the graph $mG$.
   Let $V(mG)=\{v^i: 1\leq i \leq m, v\in V(G)\}$ be its vertex set.
		Define a vertex labeling  $g\colon V(mG)\to\{1,2,\dots,m|V(G)|\}$ for $mG$ as follows: 
		\begin{equation*}
			g(v^i)=\begin{cases}
				m(f(v)-1)+i &\text{if $h(v)=1$ and $i=1, 2, \dots,  m$},\\
				m(f(v)-1)+i &\text{if $h(v)=2$ and $i=1, 2, \dots,m$},\\
                
                mf(v)+2-2i&\text{if $h(v)=3$ and $i=1, 2, \dots,\tfrac{m}{2}$},\\
                mf(v)+m+1-2i &\text{if $h(v)=3$ and $i=\tfrac{m}{2}+1, \tfrac{m}{2}+2, \dots, m$},\\
                m(f(v)-1)+i &\text{if $h(v)=l$, $l=4, 5, \dots,\tfrac{p+3}{2}$ and $i=1, 2, \dots, m$},\\
                 mf(v)+1-i &\text{if $h(v)=l$, $l=\tfrac{p+5}{2}, \tfrac{p+7}{2},\dots, p$ and $i=1,2,\dots, m$.}
			\end{cases}
		\end{equation*}
	Observe that if a vertex in $G$ is labeled with the number $t$, $t\in\{1,2,\dots,|V(G)|\}$ by $f$, then the corresponding vertices in $mG$ receive labels $\{m(t-1)+1, m(t-1)+2,\dots, mt\}$ under $g$.
    
		We now compute the weight of a vertex $u^i$ under the labeling $g$. For $1\leq i \leq {m}/{2}$ we have
		\begin{align*}
			w_{g}(u^{i})=&\sum_{v\in N_G(u)}g(v^i)=\sum_{\substack{v\in N_G(u),\\ h(v)=1}}g(v^i)+\sum_{\substack{v\in N_G(u),\\ h(v)=2}}g(v^i)+\dots+\sum_{\substack{v\in N_G(u),\\  h(v)=p}}g(v^i)\\
		=&\sum_{\substack{v\in N_G(u),\\ h(v)=1}}\left(m(f(v)-1)+i\right)+\sum_{\substack{v\in N_G(u),\\ h(v)=2}}\left(m(f(v)-1)+i \right)+\sum_{\substack{v\in N_G(u),\\ h(v)=3}}\left(mf(v)+2-2i\right)\\&+\sum_{l=4}^{\tfrac{p+3}{2}}\sum_{\substack{v\in N_G(u),\\ h(v)=l}}\left(m(f(v)-1)+i\right) +\sum_{l=\tfrac{p+5}{2}}^{p} \sum_{\substack{v\in N_G(u),\\ h(v)=l}}\left(mf(v)+1-i\right) 
          \\
          =& \ m 
          \sum_{v\in N_G(u)}f(v) +(1-m) \tfrac{p-3}{2} r+(2-2m)r
          = m w_f(u)+(1-m) \tfrac{p+1}{2} r.
		\end{align*}
And for $ {m}/{2}+1\leq i \leq m$ we obtain
    \begin{align*}
			w_{g}(u^{i})=&\sum_{v\in N_G(u)}g(v^i) =\sum_{\substack{v\in N_G(u),\\ h(v)=1}}g(v^i)+\sum_{\substack{v\in N_G(u),\\ h(v)=2}}g(v^i)+\dots+\sum_{\substack{v\in N_G(u), \\ h(v)=p}}g(v^i)\\
		=&\sum_{\substack{v\in N_G(u),\\ h(v)=1}}\left(m(f(v)-1)+i\right)+\sum_{\substack{v\in N_G(u),\\ h(v)=2}}\left(m(f(v)-1)+i\right)+\sum_{\substack{v\in N_G(u),\\ h(v)=3}}\left(mf(v)+m+1-2i\right)\\&+\sum_{l=4}^{\tfrac{p+3}{2}} \sum_{\substack{v\in N_G(u),\\ h(v)=l}}\left(m(f(v)-1)+i\right)+\sum_{l=\tfrac{p+5}{2}}^{p} \sum_{\substack{v\in N_G(u),\\ h(v)=l}}\left(mf(v)+1-i\right)
          \\
          =&\ m 
          \sum_{v\in N_G(u)}f(v) +(1-m) \tfrac{p-3}{2} r+(1-m)r
           =m w_f(u)+(1-m) \tfrac{p-1}{2}r.		\end{align*}
    These calculations follow from the fact that every vertex has exactly $r$ neighbors in each color class.
     Since $f$ is a local distance antimagic labeling of $G$,   adjacent vertices in $G$ have distinct weights, and hence   adjacent vertices in $mG$ also have distinct weights under $g$. Furthermore, from the above weight expressions it follows that the first ${m}/{2}$ copies of $G$ receive $\chi_{ld}(G)$ distinct weights, while the remaining ${m}/{2}$ copies receive another set of $\chi_{ld}(G)$ distinct weights. Therefore, $g$ is a local distance antimagic labeling for $mG$ that assigns at most $2\chi_{ld}(G)$ distinct weights to the vertices of $mG$. Hence, $\chi_{ld}(mG)\leq 2\chi_{ld}(G)$.
\end{proof}
In some cases, it is possible to obtain a better upper bound than the one presented in Lemma \ref{union-odd-nbc-even-copies}, as illustrated in the following example.
\begin{theorem}
   Let $m$ and  $t$ be even positive integers and let $n\ge 3$ be an odd positive integer. Then  $\chi_{ld}(mK_{t[n]})=\chi_{ld}(K_{t[n]})=t$.
\end{theorem}
\begin{proof}
    Let $V(K_{t[n]})=\{x_{i,j}: 1\leq i \leq n,  1\leq j\leq t\}$ and $E(K_{t[n]})=\{x_{p,j}x_{q,r}: 1\leq p, q \leq n, j\not=r\}$ be the vertex and edge set of $K_{t[n]}$, respectively. For $1 \le h \le m$, let $x_{i,j}^h$ represent the copy of vertex $x_{i,j}$ in the $h$th copy of $K_{t[n]}$ within $mK_{t[n]}$.
    
    Define a bijection $f\colon V(mK_{t[n]})\to \{1,2,\dots,mtn\}$ as follows. For $1\leq h \leq m$
    \begin{align*}
        f(x_{i,j}^h)&=
        \begin{cases}
            (i-1)m+h+3(j-1)m &\text{for $i=1,2$ and $j=1,3,\dots,t-1$},\\ 
            4m-(2h-1)+3(j-1)m &\text{for $i=3$ and $j=1,3,\dots,t-1$},\\
             4m+(j-2)3m-2(h-1) &\text{for $i=1$ and $j=2,4,\dots,t$},\\
            4m+(j-2)3m+h+(i-2)m &\text{for $i=2,3$ and $j=2,4,\dots,t$},\\ 
            (i-1)tm+(h-1)t+j &\text{for $i=4,6,\dots,n-1$ and $j=1, 2, \dots,t$},\\
            itm-(h-1)t-j+1 &\text{for $i=5,7,\dots,n$ and   $j=1, 2, \dots,  t$}.
        \end{cases}
    \end{align*}
   Fix a copy of $K_{t[n]}$ and consider its $j$th partite set.
\begin{description}
    \item[Case 1:] when $j$ is odd.
    \begin{align*}
\sum_{i=1}^3 f(x_{i,j}^h)
&= (h + 3(j-1)m) + (m + h + 3(j-1)m) + (4m - (2h-1) + 3(j-1)m) \\
&= 9jm - 4m + 1,\\
\sum_{i=4}^n f(x_{i,j}^h)
&=  \tfrac{n-3}{2} \left((n+3)tm + 1\right).
\end{align*}
Hence,
\[
\sum_{i=1}^n f(x_{i,j}^h) = 9jm - 4m + 1 +  \tfrac{n-3}{2} \left((n+3)tm + 1\right).
\]
\item[Case 2:]  when $j$ is even.

Similarly,
\[
\sum_{i=1}^3 f(x_{i,j}^h) = 9jm - 5m + 2,
\]
and thus
\[
\sum_{i=1}^n f(x_{i,j}^h) = 9jm - 5m + 2 +  \tfrac{n-3}{2} \left((n+3)tm + 1\right).
\]
\end{description}
These sums are independent of the copy index $h$. The weight of any vertex in the $j$th partite set is then given by
\begin{description}
    \item[Case 1:]  when $j$ is odd.
    \begin{align*}
    w(x_{i,j}^h)&=
    \tfrac{n-3}{2} \left((n+3)tm+1\right)(t-1)+\tfrac{t-2}{2}(-9m+3)-5m+2+9m\left(\tfrac{t(t+1)}{2}-j\right).
\end{align*}
\item[Case 2:]  when $j$ is  even.
\begin{align*}
    w(x_{i,j}^h)&=
    \tfrac{n-3}{2}\left((n+3)tm+1\right)(t-1)+ \tfrac{t-2}{2} (-9m+3)-4m+1+9m\left(\tfrac{t(t+1)}{2}-j\right).
\end{align*} 
\end{description}
The labeling $f$ assigns identical vertex weights within each of the $t$ partite sets and distinct weights between different partite sets in any copy of $K_{t[n]}$. Furthermore,  corresponding vertices across different copies have identical weights. Hence, $f$ is a local distance antimagic labeling of $mK_{t[n]}$ that induces exactly $t$ distinct vertex weights. Therefore, $\chi_{ld}(mK_{t[n]}) \leq t$. Further, as $\chi(mK_{t[n]})=t$, the equality follows.
\end{proof}
To handle the case of graphs with pendant vertices, now study the local distance antimagic chromatic number of copies of nearly neighbourhood balanced $k$-colored graphs, which were introduced in Section 3. The concept will be studied as two separate Lemmas when $k$ is even, and $k$ is odd.

\begin{lemma}\label{union-even-nnbc}
   Let $m$ be a positive integer, and let $G$ be a connected nearly
neighborhood balanced $k$-colored graph of order at least three, where $k$ is even, such that every non-pendant vertex has degree $2r$. Let $f$ be a local distance antimagic labeling of $G$ assigning
$\chi_{ld}(G)$ distinct weights to its vertices. Let $s$ be the number of distinct weights that occur both at pendant and at non-pendant vertices under $f$. Furthermore, suppose that, for every non-pendant
vertex $u$,
		\begin{equation}\label{copieseq4}
			m\left(f(u)-w_f(u)+r\right)-r>m-1  \quad \text{ or }\quad  m\left(f(u)-w_f(u)+r\right)-r<0.
		\end{equation}
        If $|N_{pen}(G)|$ denotes the number of neighbors of all pendant vertices of $G$,  
        then
        \begin{align*}
m|N_{pen}(G)|+1
&\leq \chi_{ld}(mG)
\leq \chi_{ld}(G)+(m-1)|N_{pen}(G)|+s,
&& \text{if}\ G\text{ is not a star},\\[2mm]
m
&\leq \chi_{ld}(mG)
\leq m+1,
&& \text{if}\ G\text{ is a star}.
\end{align*}
\end{lemma}
\begin{proof}
    Since $G$ is a nearly neighborhood balanced $k$-colored graph with $k$ even, it follows from Lemma \ref{nnbc-reduction} that $G$ also admits a nearly neighborhood balanced $2$-coloring. Let $h$ be such a coloring of $G$. Then, for every non-pendant vertex $u\in V(G)$, we have $$\sum_{\substack{v\in N_G(u),\\ h(v)=1}}1=\sum_{\substack{v\in N_G(u),\\  h(v)=2}}1=r.$$ 
  As $f$ is a local distance antimagic labeling of $G$ inducing $\chi_{ld}(G)$ distinct vertex weights, the pendant vertices receive $|N_{pen}(G)|$ distinct weights under \(f\), and exactly \(s\) of these weights also occur at non-pendant vertices. Hence the non-pendant vertices receive \(\chi_{ld}(G)-|N_{pen}(G)|+s\) distinct weights.
  
   Consider now the graph $mG$.
   Let $V(mG)=\{v^i;\ 1\leq i \leq m\ \text{and}\ v\in V(G)\}$ be its vertex set.
		We define a vertex labeling  $g\colon V(mG)\to\{1,2,\dots,m|V(G)|\}$  as follows. For $1\leq i \leq m,$
		\begin{equation*}
			g(v^{i})=
			\begin{cases}
				m(f(v)-1)+i &\text{if $h(v)=1$},\\
				mf(v)+1-i &\text{if $h(v)=2$}.
			\end{cases}
		\end{equation*}
We now calculate the weight of vertices in $mG$. Since each non-pendant vertex of $G$ has the same number of neighbors of each color, by the computation in Lemma \ref{union-even-nbc}  (the labeling $g$  being identical as in Lemma \ref{union-even-nbc}), for any non-pendant vertex $u\in V(G)$ we have
$$w(u^j)=m\ w_f(u)+(1-m)r.$$
Thus, $g$ is such a labeling that the corresponding non-pendant vertices in different copies receive the same weight. Moreover, since $f$ is a local distance antimagic labeling of $G$, adjacent non-pendant vertices in each copy of $G$ have distinct weights. Therefore, the non-pendant vertices of $mG$ contribute exactly $\chi_{ld}(G)-|N_{pen}(G)|+s$ distinct weights to $\chi_{ld}(mG)$.

Now consider a pendant vertex $v$ of $G$. The weight of its copies in $mG$ is given by
			$$w(v^j)=g(x^j)=\begin{cases}
				m(f(x)-1)+j &\text{if $h(x)=1$},\\
				mf(x)+1-j &\text{if $h(x)=2$},
			\end{cases}  $$ 
            where $x\in N_G(v)$.
			Clearly, the pendant vertices in corresponding copies receive distinct weights. Since each copy of $G$ contributes $|N_{pen}(G)|$ distinct weights due to pendant vertices, the graph $mG$ contributes $m|N_{pen}(G)|$ such distinct weights.
            
           Suppose in the $j$th copy of $G$, the weight of a pendant vertex $v^j$   equals the weight of its adjacent non-pendant vertex, say $u^j$. Then, for some $l=0,1,2,\dots,m-1$, we have
			\begin{align*}
				m f(u)-l&=m\cdot w_f(u)+(1-m)r, \\
				l&=m\left(f(u)-w_f(u)+r\right)-r.
			\end{align*}
			However, from inequality (\ref{copieseq4}), this equality cannot hold. Therefore, $g$ is a local distance antimagic labeling of $mG$ that induces at most $\chi_{ld}(G)+(m-1)|N_{pen}(G)|+s$ distinct weights. Hence  $\chi_{ld}(mG)\leq \chi_{ld}(G)+(m-1)|N_{pen}(G)|+s .$ In particular, if $G$ is a star, then the weight of any pendant vertex under $f$ cannot be equal to the weight of the lone non-pendant vertex and thus $s=0$. Also $|N_{pen}(G)|=1$ and $\chi_{ld}(G)=2$ and hence $\chi_{ld}(mG)\leq m+1.$
            
         To establish the lower bound, note that under any local distance antimagic labeling of $mG$, the pendant vertices in each copy of $G$ contribute $|N_{pen}(G)|$ distinct weights, and thus all pendant vertices in $mG$ contribute $m|N_{pen}(G)|$ distinct weights to $\chi_{ld}(mG)$.
         The values of all these weights, being labels of non-pendant vertices of $mG$, are  less than or equal to $mn$,  the largest label.
         If the label $mn$ is assigned to a pendant vertex, then the weight of its neighbor must exceed $mn$, since that neighbor has degree at least two. Similarly, if the label $mn$ is assigned to a non-pendant vertex, then the weight of its adjacent non-pendant vertex in the case of a non-star graph will be greater than $mn$. In the case of $G$ being a star, such a non-pendant vertex does not exist. Consequently,
\begin{align*}
m|N_{pen}(G)|+1
&\leq \chi_{ld}(mG), \hspace{0.5cm} \text{if $G$ is not a star},\\
m
&\leq \chi_{ld}(mG), \hspace{0.5cm} \text{if $G$ is a star}.
\end{align*}
\end{proof}


\begin{lemma}\label{union-odd-nnbc}
    Let $m$ be an odd positive integer and let $G$ be a connected \nnbcd of order at least three, where $k$ is odd,  such that every non-pendant vertex has degree  $kt$. Let $f$ be a local distance antimagic labeling of $G$ assigning
$\chi_{ld}(G)$ distinct weights to its vertices. Let $s$ be the number of distinct weights that occur both at pendant and at non-pendant vertices under $f$. Furthermore, suppose that, for every non-pendant
vertex $u$,
		\begin{gather}\label{copieseq5}
			m(f(u)-w_f(u))+\tfrac{(m-1)pr}{2}>m\end{gather}
            or
            \begin{gather}\label{copieseq6}
            m(f(u)-w_f(u))+\tfrac{(m-1)pr}{2}<0,
		\end{gather}
 If $|N_{pen}(G)|$ denotes the number of neighbors of all pendant vertices of $G$,   then
        \begin{align*}
m|N_{pen}(G)|+1
&\leq \chi_{ld}(mG)
\leq \chi_{ld}(G)+(m-1)|N_{pen}(G)|+s,
&& \text{if}\ G\text{ is not a star},\\[2mm]
m
&\leq \chi_{ld}(mG)
\leq m+1,
&& \text{if}\ G\text{ is a star}.
\end{align*}
\end{lemma}
\begin{proof}
    As $G$ is a nearly neighborhood balanced $k$-colored graph with $k$ odd, by Lemma \ref{nnbc-reduction}, $G$ is also a nearly neighborhood balanced $p$-colored graph, where $p$ is the least prime divisor of $k$. Since every non-pendant vertex has degree $kt$, and $p$ divides $k$, we may assume that each non-pendant vertex has degree $pr$.  Let $h$ be a nearly neighborhood balanced $p$-coloring of $G$. Then for all non-pendant vertices $u\in V(G)$ we have  $$\sum_{\substack{v\in N_G(u), \\ h(v)=1}}1=\sum_{\substack{v\in N_G(u),\\ h(v)=2}}1=\dots=\sum_{\substack{v\in N_G(u),\\ h(v)=p}}1=r.$$ 
   Since \(f\) induces \(\chi_{ld}(G)\) distinct vertex weights and exactly \(s\) distinct weights occur at both pendant and non-pendant vertices, while the pendant vertices receive \(|N_{pen}(G)|\) distinct weights, the non-pendant vertices receive \(\chi_{ld}(G)-|N_{pen}(G)|+s\) distinct weights under \(f\).  
   
  Now consider the graph $mG$ with the vertex set $V(mG)=\{v^i: 1\leq i \leq m, v\in V(G)\}$.
		Define a vertex labeling  $g\colon V(mG)\to\{1,2,\dots,m|V(G)|\}$   as follows: 
		\begin{equation*}
			g(v^i)=\begin{cases}
				m(f(v)-1)+i &\text{if $h(v)=1$ and $i=1, 2, \dots,  m$},\\
				m(f(v)-1)+\tfrac{m+1}{2}+i &\text{if $h(v)=2$ and $i=1, 2, \dots,\tfrac{m-1}{2}$},\\
                m(f(v)-1)-\tfrac{m-1}{2}+i &\text{if $h(v)=2$ and $i= \tfrac{m+1}{2}, \tfrac{m+3}{2},\dots, m$},\\
                
                mf(v)+1-2i&\text{if $h(v)=3$ and $1, 2, \dots,\tfrac{m-1}{2}$},\\
                mf(v)+m+1-2i &\text{if $h(v)=3$ and $i= \tfrac{m+1}{2}, \tfrac{m+3}{2},\dots, m$},\\
                m(f(v)-1)+i &\text{if $h(v)=t$, $t=4,5,\dots, \tfrac{p+3}{2}$ and $i=1, 2, \dots,  m$},\\
                 mf(v)+1-i &\text{if $h(v)=t$, $t=\tfrac{p+5}{2}, \tfrac{p+7}{2},\dots, p$ and $i=1, 2, \dots,  m$.}
			\end{cases}
		\end{equation*}
We now calculate the vertex weights. Since every non-pendant vertex in $G$ has an equal number of neighbors of each of the  $p$ colors, it follows from the calculation in    Lemma \ref{union-odd-nbc-odd-copies} that under the labeling $g$   we have
$$w(u^j)=m\ w_f(u)+\tfrac{(1-m)pr}{2}.$$
Hence, corresponding non-pendant vertices in different copies of $G$ have the same weights. As $f$ is a local distance antimagic labeling of $G$,   adjacent non-pendant vertices in any copy of $G$, and thus  in $mG$, receive distinct weights. Therefore, the non-pendant vertices of $mG$ contribute exactly $\chi_{ld}(G)-|N_{pen}(G)|+s$ distinct weights to $\chi_{ld}(mG)$.

Now consider a pendant vertex $v$ of $G$, and let  $x\in N_G(v)$. Then the weight of its copies in $mG$ is  
\begin{equation*}
			w(v^j)=g(x^j)=\begin{cases}
				m(f(x)-1)+j &\text{if $h(x)=1$ and $j=1, 2, \dots,  m$},\\
				m(f(x)-1)+\tfrac{m+1}{2}+j &\text{if $h(x)=2$ and $j=1, 2, \dots,  \tfrac{m-1}{2}$},\\
                m(f(x)-1)-\tfrac{m-1}{2}+j &\text{if $h(x)=2$ and $j=\tfrac{m+1}{2}, \tfrac{m+3}{2}, \dots,  m$},\\
                
                mf(x)+1-2j&\text{if $h(x)=3$ and $j=1, 2, \dots,  \tfrac{m-1}{2}$},\\
                mf(x)+m+1-2j &\text{if $h(x)=3$ and $j=\tfrac{m+1}{2}, \tfrac{m+3}{2},\dots, m$},\\
                m(f(x)-1)+j &\text{if $h(x)=t$, $t=4, 5, \dots,  \tfrac{p+3}{2}$ and $j=1, 2, \dots, m$},\\
                 mf(x)+1-j &\text{if $h(x)=t$, $t=\tfrac{p+5}{2}, \tfrac{p+7}{2},\dots, p$ and $j=1, 2, \dots,  m$.}
			\end{cases}
		\end{equation*}
			Clearly,   pendant vertices in corresponding copies receive distinct weights. Since each copy of $G$ contributes $|N_{pen}(G)|$ distinct weights due to pendant vertices, the graph $mG$ contributes $m|N_{pen}(G)|$ such weights.
            
            If in the $j$th copy of $G$,  some pendant vertex $v^j$ has the same weight as  its adjacent non-pendant vertex, say $u^j$, then for some $l\in \{0,1, \dots,m-1\}$, we have
			\begin{align*}
				m f(u)-l&=m\ w_f(u)+\tfrac{(1-m)pr}{2},\\
                l&=m(f(u)-w_f(u))+\tfrac{(m-1)pr}{2}.
			\end{align*}
			which contradicts inequalities  (\ref{copieseq5}) and (\ref{copieseq6}). Hence, $g$ is a local distance antimagic labeling of $mG$ inducing  at most $\chi_{ld}(G)+(m-1)|N_{pen}(G)|+s$ distinct weights, implying that  $\chi_{ld}(mG)\leq \chi_{ld}(G)+(m-1)|N_{pen}(G)|+s.$ In particular, if $G$ is a star, then the weight of any pendant vertex under $f$ cannot be equal to the weight of the lone non-pendant vertex and thus $s=0$. Also $|N_{pen}(G)|=1$ and $\chi_{ld}(G)=2$ and hence $\chi_{ld}(mG)\leq m+1.$

           Finally, by the same reasoning as in Lemma \ref{union-even-nnbc}, we have\begin{align*}
m|N_{pen}(G)|+1
&\leq \chi_{ld}(mG), \hspace{0.5cm} \text{if $G$ is not a star},\\
m
&\leq \chi_{ld}(mG), \hspace{0.5cm} \text{if $G$ is a star}.
\end{align*}
\end{proof}
Both bounds in Lemma \ref{union-even-nnbc} are attained by suitable stars, while for the stars covered by Lemma \ref{union-odd-nnbc}, the upper bound is attained. Priyadharshini et al. \cite{Nalliah2} completely characterized the local distance antimagic chromatic number for the union of stars.
\begin{theorem}\cite{Nalliah2}
Let $m\geq 1$, $r\geq 2$ be positive integers. Then
\[
\chi_{ld}(mK_{1,r})=
\begin{cases}
m & \text{for } r=2 \text{ and } m\geq 3,\\
m+1 & \text{for } r=2 \text{ and } m=1,2,\\
m+1 & \text{for } r\geq 3 \text{ and } m\geq 1.\\
\end{cases}
\]
\end{theorem}


\section{Conclusion}
In this article, we studied the local distance antimagic chromatic number for the direct product and the disjoint union of \nbcd graphs and examined whether the known results concerning the ordinary chromatic number of a graph remain valid for the local distance antimagic chromatic number as well. We presented several classes of \nbc graphs where the answer is in the affirmative. We also presented several results concerning the local distance antimagic chromatic number for the disjoint union of \nnbc graphs defined in Section \ref{nnbc}. The following questions naturally arise:
\begin{itemize}
    \item Characterize graphs $G$ and $H$ for which $\chi_{ld}(G\times H)=\min\{\chi_{ld}(G),\chi_{ld}(H)\}.$
    \item Characterize graphs $G$ for which $\chi_{ld}(mG)=\chi_{ld}(G)$, for any positive integer $m$.
\end{itemize}
Also, as mentioned in Section \ref{child2} at the beginning, the problem of characterizing graphs with local distance antimagic chromatic number equal to two remains open. This is an interesting problem for future research. 
\section{Acknowledgment}
  The first author is thankful to the Goa State Research Foundation (GSRF), Goa, India, for the financial support to carry out this work under the minor research project MIN2025035. The second author is thankful to the National Board for Higher Mathematics (NBHM), Govt. of India, for the financial support to carry out research work under the project Ref. No. 02011/21/2025/NBHM-RP/RD-II/9821. The third author is thankful to the Slovak Research and Development Agency under the contract No. APVV-23-0191 and to VEGA 1/0243/23. \\
The authors thank the referees for their valuable suggestions/comments, which helped to improve the overall presentation of the article.

 \end{document}